\documentclass[11pt]{amsart}

\usepackage{epigamath-voisin}

\usepackage[english]{babel}

\numberwithin{equation}{section}

\usepackage[shortlabels]{enumitem}
\setlist[enumerate,1]{label={\rm(\arabic*)}, ref={\rm\arabic*}} 
\setlist[enumerate,2]{label={\rm(\alph*)}, ref={\rm\alph*}} 

\input xy
\xyoption{all}

\usepackage{amscd}
\usepackage{amssymb}

\usepackage{amsthm,array,amssymb,amscd,amsfonts,latexsym, url}
\usepackage{amsmath}
\usepackage{graphicx,epsfig}

\usepackage{fancyvrb}
\usepackage{fvextra}

\newtheorem{theoreme}{Theorem}[section]
\newtheorem{lemme}[theoreme]{Lemma}

\newtheorem{proposition}[theoreme]{Proposition}

\theoremstyle{definition}

\theoremstyle{remark}
\newtheorem{remarque}[theoreme]{Remark}

\DeclareMathOperator{\Alb}{Alb}
\DeclareMathOperator{\alb}{alb}
\DeclareMathOperator{\CH}{CH}
\DeclareMathOperator{\Sym}{Sym}
\DeclareMathOperator{\eq}{eq}
\DeclareMathOperator{\Span}{Span}
\DeclareMathOperator{\Pic}{Pic}
\DeclareMathOperator{\id}{id}
\DeclareMathOperator{\Ker}{Ker}
\DeclareMathOperator{\Coker}{Coker}
\DeclareMathOperator{\coker}{coker}
\DeclareMathOperator{\red}{red}
\DeclareMathOperator{\rank}{rank}
\DeclareMathOperator{\Hom}{Hom}
\DeclareMathOperator{\torsion}{torsion}
\DeclareMathOperator{\Fl}{Fl}
\DeclareMathOperator{\Gr}{Gr}
\DeclareMathOperator{\ev}{ev}
\DeclareMathOperator{\topp}{top}
\DeclareMathOperator{\univ}{univ}
\DeclareMathOperator{\pt}{pt}
\DeclareMathOperator{\Line}{line}

\def\isolow{\vbox to 0pt{\vss\hbox{$\scriptstyle\sim$}\vskip-1.8pt}}
\newcommand{\isor}{\xrightarrow{\;\isolow\;}}
\newcommand{\shortisor}{\xrightarrow{\isolow}}

\def\shortbar{%
\smash{\scalebox{0.4}[1.0]{$-$}}}
\makeatother
     
\makeatletter
\newcommand{\xdashrightarrow}[2][]{\ext@arrow 0359\rightarrowfill@@{#1}{#2}}
\def\rightarrowfill@@{\arrowfill@@\relax\shortbar\dashrightarrow}
\def\arrowfill@@#1#2#3#4{%
  $\m@th\thickmuskip0mu\medmuskip\thickmuskip\thinmuskip\thickmuskip
   \relax#4#1
   \xleaders\hbox{$#4#2$}\hfill
   #3$%
}
\makeatother

\newcommand{\longdashrightarrow}{\xdashrightarrow{\hspace{9pt}}}

\newcommand{\xrightarrowdbl}[1]{%
  \xrightarrow{#1}\mathrel{\mkern-14mu}\rightarrow
}

\YearArticle{2023} \EpigaArticleNr{8} \ReceivedOn{January 13, 2023}
\InFinalFormOn{June 5, 2023}
\AcceptedOn{June 27, 2023}

\title{Remarks on the geometry of the variety of planes\\ of a cubic fivefold}
\titlemark{The geometry of the variety of planes of a cubic fivefold}
\author{Ren\'e Mboro}
\address{UMiami Miami, HSE Moscow,\\
Institute of Mathematics and Informatics, Bulgarian Academy of Sciences,
Acad.\ G.~Bonchev Str.\ bl.~8, 1113, Sofia, Bulgaria}
\email{rene.mboro@polytechnique.edu}

\authormark{R.~Mboro}

\AbstractInEnglish{This note presents some properties of the variety of planes $F_2(X)\subset G(3,7)$ of a cubic $5$-fold $X\subset \mathbb P^6$. A cotangent bundle exact sequence is first derived from the remark made by Iliev and Manivel
  that $F_2(X)$ sits as a Lagrangian subvariety of the variety of lines of a cubic $4$-fold, which is a hyperplane section of $X$. Using the sequence, the Gauss map of $F_2(X)$ is then proven to be an embedding. The last section is devoted to the relation between the variety of osculating planes of a cubic $4$-fold and the variety of planes of the associated cyclic cubic $5$-fold.}



\acknowledgement{The author was partially supported by an NSF Grant, a Simons
  Investigator Award HMS, a Simons Collaboration Award HMS, the HSE
  University Basic Research Program and the Ministry of Education and
  Science of the Republic of Bulgaria through the Scientific Program
  “Enhancing the Research Capacity in Mathematical Sciences (PIKOM)”
  No. DO1-67/05.05.2022.}

\dedication{Dedicated to Claire Voisin on the occasion of her 60th birthday}

\begin{document}

%%%%%%%%%%%%%%%%%%%%%%%%%%%%%%%
% Title page
%%%%%%%%%%%%%%%%%%%%%%%%%%%%%%%

%\removeabove{}
%\removebetween{}
%\removebelow{}

\maketitle

\begin{prelims}

\DisplayAbstractInEnglish

\bigskip

\DisplayKeyWords

\medskip

\DisplayMSCclass

%\bigskip

%\languagesection{Fran\c{c}ais}

%\bigskip

%\DisplayTitleInFrench

%\medskip

%\DisplayAbstractInFrench

\end{prelims}

%%%%%%%%%%%%%%%%%%%%%
% Table of Contents
%%%%%%%%%%%%%%%%%%%%%

\newpage

\setcounter{tocdepth}{1}

\tableofcontents

%%%%%%%%%%%%%%%%%%%%%
% Content begins here
%%%%%%%%%%%%%%%%%%%%%

\section{Introduction}
To understand the topology and the geometry of smooth complex hypersurfaces $X\subset \mathbb P(V^*)\simeq\mathbb P^{n+1}$, various auxiliary manifolds have been introduced in the past century, of which the intermediate Jacobian
$$
J^n(X):=(H^{k-1,k+2}(X)\oplus\cdots\oplus H^{0,n})/H^n(X,\mathbb Z)_{/ \torsion}
$$
when $n=2k+1$ is odd is one of the most widely known since the seminal work of Clemens--Griffiths (\cite{Cl-Gr}) on the cubic $3$-fold.

 Cubic $5$-folds are classically (\textit{cf.}~\cite{griffiths_periods}) known to be the only hypersurfaces of dimension greater than $3$ for which the intermediate Jacobian, which is in general just a (polarised) complex torus, is a (non-trivial) principally polarised abelian variety.

 Another interesting series of varieties classically associated to $X$ are the varieties $F_m(X)\subset G(m+1,V)$ of $m$-planes contained in $X$.
 
 Starting from Collino (\cite{Coll_cub}), some properties of the variety of planes $F_2(X)\subset G(3,V)$ of a cubic $5$-fold $X$ have been studied in connection with the $21$-dimensional intermediate Jacobian $J^5(X)$. In \textit{loc.\ cit.}, the following is proven.

\begin{theoreme}\label{thm_Collino_intro} For a general cubic $X\subset \mathbb P(V^*)\simeq \mathbb P^6$, $F_2(X)$ is a smooth irreducible surface, and the Abel--Jacobi map 
% query en-dash
  of the family of planes\, $\Phi_{\mathcal P}\colon F_2(X)\rightarrow J^5(X)$ is an immersion; \textit{i.e.}, the associate tangent map is injective and induces an isomorphism of abelian varieties
  $$
  \phi_{\mathcal P}\colon \Alb(F_2(X))\isor J^5(X),
  $$
  where $\mathcal P\in {\CH}^5(F_2(X)\times X)$ is the universal plane over $F_2(X)$. Equivalently, $q_*p^*\colon H^3(F_2(X),\mathbb Z)_{/ \torsion}\rightarrow H^5(X,\mathbb Z)$ is an isomorphism of Hodge structures, where the maps are defined by
  $$\xymatrix{\mathcal P\ar[r]^q\ar[d]^p &X\\ 
    F_2(X). &}
$$
\end{theoreme}

In the present note, we investigate some additional properties of $F_2(X)$.

In the first section, we establish the following cotangent bundle exact sequence. 

\begin{theoreme}\label{thm_1} Let $X\subset \mathbb P(V^*)$ be a smooth cubic $5$-fold for which $F_2(X)$ is a smooth irreducible surface. Then the cotangent bundle $\Omega_{F_2(X)}$ fits in the exact sequence
  \begin{equation}\label{ex_seq_tgt_bundle_seq}
    0\longrightarrow \mathcal Q_{3}^*|_{F_2(X)}\longrightarrow {\Sym}^2\mathcal E_{3}|_{F_2(X)}\longrightarrow \Omega_{F_2(X)}\longrightarrow 0, 
\end{equation} 
where the tautological rank $3$ quotient bundle $\mathcal E_3$ and the other bundle appear in the exact sequence
\begin{equation}\label{ex_seq_def_taut_3} 0\longrightarrow \mathcal Q_3\longrightarrow V^*\otimes \mathcal O_{G(3,V)}\longrightarrow \mathcal E_3\longrightarrow 0
\end{equation}
and the first map $($of \eqref{ex_seq_tgt_bundle_seq}$)$ is the contraction with an equation ${\eq}_X\in {\Sym}^3V^*$ defining $X$, \textit{i.e.}\
% query 
for any $[P]\in F_2(X)$, $v\mapsto {\eq}_X(v,\cdot,\cdot)|_{P}$.
\end{theoreme}  

Classically associated to the Albanese map $\alb_{F_2}\colon F_2(X)\rightarrow \Alb(F_2(X))$ of $F_2(X)$, there is the Gauss map 
\begin{alignat*}{2}
  \mathcal G\colon & \alb_{F_2}(F_2(X))  &\ \longdashrightarrow \ & G\left(2, T_{\Alb(F_2(X)),0}\right)\\
  & t &\ \longmapsto \ & T_{\alb_{F_2}(F_2(X))-t,0}
\end{alignat*}
where $\alb_{F_2}(F_2(X))-t$ designates the translation of $\alb_{F_2}(F_2(X))\subset {\Alb}(F_2(X))$ by $-t\in {\Alb}(F_2(X))$. The map $\mathcal G$ is defined on the smooth locus of $\alb_{F_2}(F_2(X))$.

 In the second section of the note, we prove the following. 

\begin{theoreme}\label{thm_gauss_map} The Albanese map is an embedding. In particular, the Gauss map is defined everywhere.
Moreover, $\mathcal G$ is an embedding, and its composition with the Pl\"ucker embedding
  $$
  G\left(2,_{\Alb(F_2(X)),0}\right)\simeq G\left(2,H^0\left(\Omega_{F_2}\right)^*\right)\subset \mathbb P\left(\bigwedge^2 H^0\left(\Omega_{F_2(X)}\right)^*\right)
  $$
  is the composition of the degree $3$ Veronese of the natural embedding $F_2(X)\subset G(3,V)\subset \mathbb P(\bigwedge^3V^*)$ followed by a linear projection.
\end{theoreme}

The last section is concerned with some properties of the variety of osculating planes of a cubic $4$-fold, namely 
\begin{equation}\label{def_var_of_oscul_planes} 
  F_0(Z):=\{[P]\in G(3,H),\ \exists \ell\subset P\ \text{line\ s.t.}\ P\cap Z=\ell\ \text{(set-theoretically)}\},
\end{equation}
where $Z\subset \mathbb P(H^*)\simeq \mathbb P^5$ is a smooth cubic $4$-fold containing no plane.

This variety admits a natural projection to the variety of lines $F_1(Z)$ of $Z$ whose image (under that projection) has been studied, for example, in \cite{GK_geom_lines}. The interest of the authors there for the variety $F_0(Z)$ stems from its image in $F_1(Z)$ being the fixed locus of the Voisin self-map of $F_1(Z)$ (see \cite{Voisin_map}), a map that plays an important role in the understanding of algebraic cycles on the hyper-K\"ahler $4$-fold $F_1(Z)$ (see for example \cite{shen-vial}).

In \cite{GK_geom_lines}, it is proven that for $Z$ general, $F_0(Z)$ is a smooth irreducible surface, and some of its invariants are computed.

We compute some more invariants of $F_0(Z)$ using its link with the variety of planes $F_2(X_Z)$ of the associated cyclic cubic $5$-fold: to a smooth cubic $4$-fold $Z=\{\eq_Z=0\}\subset \mathbb P^5$, one can associate the cubic $5$-fold $X_Z=\{X_6^3+\eq_Z(X_0,\dots,X_5)\}$ which (by linear projection) is the degree $3$ cyclic cover of $\mathbb P^5$ ramified over $Z$. 

\begin{theoreme}\label{thm_sum_up_var_oscul_planes} For $Z$ general, $F_0(Z)$ is a smooth irreducible surface, and
  \begin{enumerate}
\item $F_2(X_Z)$ is a degree $3$ \'etale cover of $F_0(Z)$,
\item $b_1(F_0(Z))=0$, $h^2(\mathcal O_{F_0(Z)})=1070$, $h^1(\Omega_{F_0(Z)})=2207$,
\item $\Im(F_0(Z)\rightarrow F_1(Z))$ is a $($non-normal\,$)$ Lagrangian surface of\, $F_1(Z)$.
  \end{enumerate}
\end{theoreme}

\begin{remarque} As mentioned by the referee and Frank Gounelas, in \cite{GK_geom_lines}, it is proven  that $[\Im(F_0(Z)\rightarrow F_1(Z))]=21[F_1(Z\cap H)]$ in $\CH_2(F_1(Z))$, where $Z\cap H$ is a cubic $3$-fold obtained as a general hyperplane section, which implies that $[\Im(F_0(Z)\rightarrow F_1(Z))]$ is Lagrangian (see \cite[Lemma 6.4.5]{Huy_cub}, for example).
\end{remarque}

\section*{Acknowledgments}
I would like to thank Hsueh-Yung Lin for pointing me to the article \cite{Iliev-Manivel_cub_hyp_int_syst} some years ago. I would like to also thank  Pieter Belmans for explaining how to use Sage to decompose the tensor powers of $\mathcal E_3$ into irreducible modules and the anonymous referee for their remarks.
% query gender neutral

Finally, I am grateful to the gracious Lord for His care.

\section{Cotangent bundle exact sequence}
Let $X\subset \mathbb P(V^*)\simeq \mathbb P^6$ be a smooth cubic $5$-fold. Its variety of planes $F_2(X)\subset G(3,V)$ is the zero locus of the section of $\Sym^3\mathcal E_3$ (where $\mathcal E_3$ is defined by (\ref{ex_seq_def_taut_3})) induced by an equation $\eq_X\in H^0(\mathcal O_{\mathbb P^6}(3))$ of $X$.

Let us gather some basic properties of $F_2(X)$ before proving Theorem \ref{thm_1}.
 
It is proven in \cite[Proposition 1.8]{Coll_cub} that $F_2(X)$ is connected for any $X$, so that by Bertini-type theorems, for $X$ general, $F_2(X)$ is a smooth irreducible surface.

As such an $F_2(X)$ is cut out of $G(3,V)$ by a regular section of the rank $10$ vector bundle $\Sym^3\mathcal E_3$, the Koszul resolution says that the structure sheaf $\mathcal O_{F_2(X)}$ is quasi-isomorphic to the complex 
\begin{equation}\label{ex_seq_koszul_resol} 
0\longrightarrow \wedge^{10}\Sym^3\mathcal E_3^*\longrightarrow\wedge^9\Sym^3\mathcal E_3^*\longrightarrow \cdots\longrightarrow \Sym^3\mathcal E_3^*\longrightarrow \mathcal O_{G(3,V)}\longrightarrow 0, 
\end{equation}
where the differentials are given by the section of $\Sym^3\mathcal E_3$. By the adjunction formula, 
$$
K_{F_2(X)}\simeq K_{G(3,V)}\otimes \det(\Sym^3\mathcal E_{3}|_{F_2(X)})\simeq \mathcal O_{G(3,V)}(3)|_{F_2(X)}:=\mathcal O_{F_2(X)}(3).
$$

Theorem \ref{thm_Collino_intro} (see also Theorem \ref{thm_descrip_h_1_and_wedge} below) implies that $h^{1,0}(F_2(X))=h^0(\Omega_{F_2(X)})= h^{2,3}(X)=21$, and we can use software to compute the other Hodge numbers (see also \cite{gammel}). We use the package Schubert2 of Macaulay2:
\begin{enumerate}
\item The Koszul resolution of $\mathcal O_{F_2(X)}$ gives $\chi(\mathcal O_{F_2(X)})=\sum_{i=0}^{10}(-1)^i\chi(\wedge^i\Sym^3\mathcal E_3^*)$. We can get the result $\chi(\mathcal O_{F_2(X)})=3213$ using the following code: 
\begin{verbatim}
loadPackage "Schubert2"
G=flagBundle{4,3}
(Q,E)= bundles G
F=symmetricPower(3,dual(E))
chi(exteriorPower(0,F))-chi(exteriorPower(1,F))+chi(exteriorPower(2,F))
-chi(exteriorPower(3,F))+chi(exteriorPower(4,F))-chi(exteriorPower(5,F))
+chi(exteriorPower(6,F))-chi(exteriorPower(7,F))+chi(exteriorPower(8,F))
-chi(exteriorPower(9,F))+chi(exteriorPower(10,F))
\end{verbatim}
Then we get $h^{0,2}(F_2(X))=\chi(\mathcal O_{F_2(X)})-1+h^{0,1}(F_2(X))=3233$.

\item Next, Noether's formula reads $\chi_{\topp}(F_2(X))=12\chi(\mathcal O_{F_2(X)})-\int_{F_2(X)}c_1(K_{F_2(X)})^2$, and as 
  $$
    \begin{aligned}
      \int_{F_2(X)}c_1\left(K_{F_2(X)}\right)^2 &=\int_{F_2(X)}c_1\left(\mathcal O_{G(3,V)}(3)|_{F_2(X)}\right)^2\\
&=\int_{G(3,V)}\left[F_2(X)\right]\cdot c_1\left(\mathcal O_{G(3,V)}(3)\right)^2\\
      &= 9\int_{G(3,V)}c_{10}\left(\Sym^3\mathcal E_3\right)\cdot c_1\left(\mathcal O_{G(3,V)}(1)\right)^2,
    \end{aligned}
  $$
the number $\int_{F_2(X)}c_1(K_{F_2(X)})^2=3^2\times 2835=25515$ can be obtained using the code
\begin{verbatim}
loadPackage "Schubert2"
G=flagBundle{4,3}
(Q,E)= bundles G
F=symmetricPower(3,E)
cycle=chern(1,exteriorPower(3,E))*chern(1,exteriorPower(3,E))*chern(10,F)
integral cycle
\end{verbatim}
 Then we get $b_2(F_2(X))=\chi_{\topp}(F_2(X))-2+2b_1(F_2(X))=13041-2+4\times 21=13123$ and $h^{1,1}(F_2(X))=b_2(F_2(X))-2h^{0,2}(F_2(X))=6657$.
\end{enumerate}

Associated to $X$, there is also its variety of lines $F_1(X)\subset G(2,V)$. It is a smooth Fano variety of dimension $6$ which is cut out by a regular section of $\Sym^3\mathcal E_2$, where $\mathcal E_2$ is the tautological rank $2$ quotient bundle appearing in an exact sequence
$$
0\longrightarrow \mathcal Q_2\longrightarrow V^*\otimes \mathcal O_{G(2,V)}\longrightarrow \mathcal E_2\longrightarrow 0.
$$
Let us examine the relation between the two auxiliary varieties by introducing the flag variety 
$$
\xymatrix{\Fl(2,3,V) \ar[d]_t\ar[r]^e &\Gr(2,V)\\
  \Gr(3,V), &}
$$
where $t\colon\Fl(2,3,V)\simeq \mathbb P(\wedge^2 \mathcal E_3)\rightarrow \Gr(3,V)$ and $e\colon\Fl(2,3,V)\simeq \mathbb P(\mathcal Q_2)\rightarrow \Gr(2,V)$. For the tautological quotient line bundles, we have $\mathcal O_{t}(1)\simeq e^*\mathcal O_{\Gr(2,V)}(1)$ and $\mathcal O_{e}(1)\simeq t^*\mathcal O_{\Gr(3,V)}(1)\otimes e^*\mathcal O_{\Gr(2,V)}(-1)$.

On $\Fl(2,3,V)$, the relation between the two tautological bundles is given by the exact sequence
\begin{equation}\label{ex_seq_taut_bundles_2_3}
0\longrightarrow e^*\mathcal O_{G(2,V)}(-1)\otimes t^*\mathcal O_{G(3,V)}(1)\longrightarrow t^*\mathcal E_3\longrightarrow e^*\mathcal E_2\longrightarrow 0. 
\end{equation}

We can restrict the flag bundle to get 
$$
\xymatrix{
  \mathbb P_{F_2}:=\mathbb P\left(\wedge^2\mathcal E_{3}|_{F_2(X)}\right)\ar[r]^(.65){e_{F_2}}\ar[d]_{t_{F_2}} &F_1(X)\\
  F_2(X)\rlap{.}&}
$$ 

We have the following property. 

\begin{proposition}\label{prop_immersion_plan_lines}
  The tangent map $T e_{F_2}$ of $e_{F_2}$ is injective; \textit{i.e.}, $e_{F_2}$ is an immersion. Moreover, the ``normal bundle'' $N_{\mathbb P_{F_2}/F_1(X)}:= e_{F_2}^*T_{F_1(X)}/T_{\mathbb P_{F_2}}$ of\, $\mathbb P_{F_2}$ admits the following description:
\begin{equation}\label{ex_seq_tgt_bundle_part1}0\longrightarrow t_{F_2}^*(\mathcal Q_{3}^*|_{F_2(X)})\otimes \mathcal O_e(1)\longrightarrow t_{F_2}^*\Sym^2\mathcal E_3\otimes \mathcal O_e(1)\longrightarrow N_{\mathbb P_{F_2}/F_1(X)}\longrightarrow 0. 
\end{equation}
\end{proposition}

\begin{proof} (1)  Let us first prove that $e_{F_2}$ is an immersion. Let us recall the natural isomorphism between the two presentations of the tangent space of $\Fl(2,3,V)$: looking at $t$, we can write
  $$
  T_{\Fl(2,3,V), ([\ell],[P])}\simeq \Hom(\langle P\rangle,V/\langle P\rangle)\oplus \Hom(\langle \ell\rangle,\langle P\rangle/\langle \ell\rangle),
  $$
  and looking at $e$, we have
  $$
  T_{\Fl(2,3,V), ([\ell],[P])}\simeq \Hom(\langle\ell\rangle,V/\langle\ell\rangle)\oplus \Hom(\langle P\rangle/\langle\ell\rangle,V/\langle P\rangle),
  $$
  where we denote by $\langle K\rangle\subset V$ the linear subspace whose projectivisation is $K\subset \mathbb P(V^*)$. For a given decomposition $\langle P\rangle\simeq \langle\ell\rangle\oplus \langle P\rangle/\langle\ell\rangle$, the isomorphism takes the following form:  
\begin{align*}
  \Hom(\langle P\rangle,V/\langle P\rangle)\oplus \Hom(\langle\ell\rangle,\langle P\rangle/\langle\ell\rangle) &\ \longrightarrow \ \Hom(\langle\ell\rangle,V/\langle\ell\rangle)\oplus \Hom(\langle P\rangle/\langle\ell\rangle,V/\langle P\rangle).\\
(f,\ g) &\ \longmapsto \ \left(f|_{\langle \ell\rangle}+g,\ f|_{\langle P\rangle/\langle\ell\rangle}\right)
\end{align*}
Notice that, by definition, we have $\Im(f)\cap \Im(g)=\{0\}$, so that in proving that $T_{([\ell],[P])}e_{F_2}$ is injective, we can examine the two components separately. 

Now we have the exact sequence
$$
0\longrightarrow N_{\ell/P}\longrightarrow N_{\ell/X}\longrightarrow N_{P/X}|_{\ell}\longrightarrow 0,
$$
 from which we get
\begin{equation}
\label{ex_seq_normal_bdle_coh_lines}
0\longrightarrow \underset{\simeq \langle\ell\rangle^*}{H^0(\mathcal O_{\ell}(1))}\longrightarrow H^0(N_{\ell/X})\longrightarrow H^0(N_{P/X}|_{\ell})\longrightarrow 0=H^1(\mathcal O_\ell(1)),
\end{equation}
and we have $T_{F_1(X),[\ell]}\simeq H^0(N_{\ell/X})$.

A linear form on $P$ defining $\ell$ is given by any generator of $(\langle P\rangle/\langle\ell\rangle)^*\subset \langle P\rangle^*$, so that
$$
T_{\mathbb P(\wedge^2 \mathcal E_{3}|_{F_2(X)}),([\ell],[P])}\simeq \underbrace{T_{F_2(X),[P]}}_{\simeq H^0(N_{P/X})}\oplus \underbrace{\langle P\rangle^*/(\langle P\rangle/\langle\ell\rangle)^*}_{\simeq \langle\ell\rangle^*}.
$$
The second summand is readily seen to inject into $T_{F_1(X),\langle\ell\rangle}$ by (\ref{ex_seq_normal_bdle_coh_lines}).  

Next, we have the exact sequence
$$
0\longrightarrow N_{P/X}(-1)\longrightarrow N_{P/X}\longrightarrow N_{P/X}|_{\ell}\longrightarrow 0,
$$
which gives rise to 
\begin{equation}\label{ex_seq_normal-1_coh}
0\longrightarrow H^0\left(N_{P/X}(-1)\right)\longrightarrow H^0\left(N_{P/X}\right)\longrightarrow H^0\left(N_{P/X}|_{\ell}\right)\longrightarrow H^1\left(N_{P/X}(-1)\right)\longrightarrow H^1\left(N_{P/X}\right).
\end{equation}
To prove that $T_{([\ell],[P])}e_{F_2}$ is injective, it is thus sufficient to prove that $H^0(N_{P/X}(-1))=0$.

Consider the exact sequence 
\begin{equation}\label{ex_seq_nomal_plane} 0\longrightarrow N_{P/X}\longrightarrow \underbrace{N_{P/\mathbb P^6}}_{\simeq(V/\langle P\rangle)\otimes \mathcal O_P(1)} \overset{\alpha}{\longrightarrow} \underbrace{N_{X/\mathbb P^6}|_{P}}_{\simeq \mathcal O_P(3)}\longrightarrow 0.
\end{equation}
Up to a projective transformation, we can assume $P=\{X_0=\cdots =X_3=0\}$, so that $\eq_X$ has the following form:
\begin{equation}\label{normal_form_0}
    X_0Q_0 + X_1Q_1 + X_2Q_2 +X_3Q_3 + \sum_{i=4}^6X_iD_i(X_0,X_1,X_2,X_3) + R(X_0,X_1,X_2,X_3)
\end{equation}
where $R$ is a homogeneous cubic polynomial, the $D_i$, $4\leq i\leq 6$, are homogeneous quadratic polynomials in the variables $(X_k)_{k\leq 3}$ and the $Q_i$, $0\leq i\leq 3$, are homogeneous quadratic polynomials in $(X_i)_{4\leq i\leq 6}$. With this notation, $X$ is smooth along $P$ if and only if $\Span((Q_{i}|_{P})_{i=0,\dots,3})$ is base-point-free. We recall the following result found in \cite[Proposition 1.2 and Corollary 1.4]{Coll_cub}. 

\begin{proposition}\label{prop_result_Collino_smoothness} For $X$ smooth along $P$, the following properties are equivalent:
  \begin{enumerate}
  \item The variety $F_2(X)$ is smooth at $[P]$.
  \item The set $(Q_0,\dots,Q_3)$ is linearly independent.
  \item The map $H^0(\alpha)\colon H^0(N_{P/\mathbb P^6})\simeq (V/\langle P\rangle)\otimes H^0(\mathcal O_P(1))\rightarrow H^0(N_{X/\mathbb P^6}|_{P})\simeq H^0(\mathcal O_P(3))$, $(L_0,\dots,L_3)\mapsto \sum_iL_iQ_i$ is surjective.
    \end{enumerate}
\end{proposition}

Now tensoring (\ref{ex_seq_nomal_plane}) by $\mathcal O_P(-1)$, we get the  long exact sequence
 \begin{equation}\label{ex_seq_normal_plane-1_coh}
0\longrightarrow H^0(N_{P/X}(-1))\longrightarrow V/\langle P\rangle\xrightarrow{H^0(\alpha(-1))} H^0(\mathcal O_P(2))\longrightarrow H^1(N_{P/X}(-1))\longrightarrow 0=H^1(\mathcal O_P)^{\oplus 4}.
\end{equation}
The map $H^0(\alpha(-1))$ is given by the quadrics $(Q_0,\dots,Q_3)$. As $F_2(X)$ is smooth by assumption, the latter are linearly independent; thus $H^0(\alpha(-1))$ is injective; \textit{i.e.}, we have $H^0(N_{P/X}(-1))=0$. In particular, $H^0(N_{P/X})\subset H^0(N_{P/X}|_{\ell})$; hence, looking at (\ref{ex_seq_nomal_plane}) and (\ref{ex_seq_normal_bdle_coh_lines}), we see that $T_{([\ell],[P])}e_{F_2}$ is injective.

(2)  We want now to establish the exact sequence (\ref{ex_seq_tgt_bundle_part1}). Pulling back the natural exact sequence of locally free sheaves, we get the commutative diagram
$$\xymatrix{0 \ar[r] &T_{\mathbb P_{F_2}}\ar[r]\ar[d]^{Te_{F_2}} & T_{\Fl(2,3,V)}|_{\mathbb P_{F_2}}\ar[r]\ar[d]^{Te|_{\mathbb P_{F_2}}} &(t^*\Sym^3\mathcal E_3)|_{\mathbb P_{F_2}} \ar[r]\ar[d]^{\overline{Te|_{\mathbb P_{F_2}}}} &0\\
0 \ar[r] &e_{F_2}^*T_{F_1(X)} \ar[r] &e_{F_2}^*T_{\Gr(2,V)}|_{F_1(X)}\ar[r] &e_{F_2}^*\Sym^3\mathcal E_{2}|_{F_1(X)}\ar[r] &0\rlap{,}}$$
which by the snake lemma yields
$$
0 \longrightarrow \Ker\left(Te|_{\mathbb P_{F_2}}\right)\longrightarrow \Ker\left(\overline{Te|_{\mathbb P_{F_2}}}\right)\longrightarrow \coker\left(Te_{F_2}\right)\longrightarrow 0.
$$
By the definition of the normal bundle, we get $\coker(Te_{F_2})\simeq N_{\mathbb P_{F_2}/F_1(X)}$. The restriction of the  exact sequence of locally free sheaves
$$
0\longrightarrow T_{\Fl(2,3,V)/\Gr(2,7)}\longrightarrow T_{\Fl(2,3,V)}\longrightarrow e^*T_{\Gr(2,V)}\longrightarrow 0
  $$
   still being exact, we get $\ker(Te|_{\mathbb P_{F_2}})\simeq T_{\Fl(2,3,V)/\Gr(2,V)}|_{\mathbb P_{F_2}}$. The relative tangent bundle appears in the exact sequence:
  $$
   0\longrightarrow \mathcal O_{\Fl(2,3,V)}\longrightarrow e^*V/\mathcal E_2^*\otimes \mathcal O_e(1)\longrightarrow T_{\Fl(2,3,V)/\Gr(2,V)}\longrightarrow 0.
  $$
  The sequence (\ref{ex_seq_taut_bundles_2_3}) also yields
  $$
  0\longrightarrow t^*\mathcal O_{\Gr(3,V)}(-1)\otimes e^*\mathcal O_{\Gr(2,V)}(1)\longrightarrow V/\mathcal E_2^*\longrightarrow V/\mathcal E_3^*\longrightarrow 0,
  $$
  from which,  after twisting that last sequence by $\mathcal O_e(1)$, we get $T_{\Fl(2,3,V)/\Gr(2,V)}|_{\mathbb P_{F_2}}\simeq t_{F_2}^*V/\mathcal E_3^*\otimes \mathcal O_e(1)$.
 
  Next, taking the symmetric power of (\ref{ex_seq_taut_bundles_2_3}) we get the  exact sequence
  $$
  0\longrightarrow e^*\mathcal O_{\Gr(2,V)}(-1)\otimes t^*\mathcal O_{\Gr(3,V)}(1)\otimes t^*\Sym^2\mathcal E_3\longrightarrow t^*\Sym^3\mathcal E_3\longrightarrow e^*\Sym^3\mathcal E_2\longrightarrow 0,
    $$
  so that $\ker(\overline{Te|_{\mathbb P_{F_2}}})\simeq (e^*\mathcal O_{\Gr(2,V)}(-1)\otimes t^*\mathcal O_{\Gr(3,V)}(1)\otimes t^*\Sym^2\mathcal E_3)|_{\mathbb P_{F_2}}$. Putting everything together, we get the desired exact sequence.
\end{proof}

For any plane $P_0\subset X$, looking for example at the associated quadric bundle 
$$
\xymatrix{\widetilde{X_{P_0}}\ar[rd]^{\tilde\gamma}\ar@{^{(}->}[r] &\mathbb P(\mathcal E_4)\ar[d]^{\gamma}\\
  &B\rlap{,}}
$$
where $B\simeq \{[\Pi]\in G(4,V),\ P_0\subset \Pi\}\simeq \mathbb P^3$, $\mathcal E_4\simeq \langle P\rangle^*\otimes \mathcal O_{\mathbb P^3}\oplus \mathcal O_{\mathbb P^3}(1)$ and $\widetilde{X_{P_0}}\in |\mathcal O_{\gamma}(2)\otimes \gamma^*\mathcal O_{\mathbb P^3}(1)|$, we see that the locus of quadrics of rank at most $2$ has codimension (at most) $\binom{4-2+1}{2}=3$. Moreover, by the Harris--Tu formula (\cite[Theorem 1 and Theorem 10]{HT_degener}), there are (at least) $2\left|\begin{smallmatrix}
c_2(\mathcal E_4\otimes L) &c_3(\mathcal E_4\otimes L)\\
c_0(\mathcal E_4\otimes L) &c_1(\mathcal E_4\otimes L 
\end{smallmatrix}\right|=31$ of these quadrics (where $L$ has to be thought of as a formal square root of $\mathcal O_{\mathbb P^3}(1)$).

In particular, the locus $\Gamma=\{([\ell],[P])\in \mathbb P_{F_2},\ \exists [P']\neq [P],\ ([\ell],[P'])\in \mathbb P_{F_2}\}$ has codimension $2$ in $\mathbb P_{F_2}$ (above the general plane $[P]\in F_2(X)$, there are finitely many lines that belong to another planes $P'\subset X$).

To any hyperplane $H\subset \mathbb P(V^*)$ such that $Y:=X\cap H$ is a smooth cubic $4$-fold containing no plane, we can attach the morphism $j_H\colon F_2(X)\rightarrow F_1(Y)$ defined by $[P]\mapsto [P\cap H]$.

The subvariety $F_1(Y)\subset F_1(X)$ is the zero locus of the regular section of $\mathcal E_{2}|_{F_1(X)}$ induced by the equation of $H\subset \mathbb P(V^*)$. For any such $Y$ (containing no plane), $e^{-1}(F_1(Y))$ is obviously a section $Z_H$ of $\mathbb P_{F_2}\rightarrow F_2(X)$, $[P]\mapsto ([P\cap H],[P])$. The smooth surface $Z_H\simeq F_2(X)$ is thus the zero locus of a regular section of $e_{F_2}^*\mathcal E_{2}|_{F_1(X)}$. By Bertini-type theorems, for $H$ general, $Z_H\cap \Gamma$ is $0$-dimensional.

As a result, as noticed in \cite[Proposition 7]{Iliev-Manivel_cub_hyp_int_syst} (the published version corrects the preprint, in which it is wrongly claimed that $j_H$ is an embedding, as underlined in \cite{Huy_cub}), $j_H\colon Z_H\simeq F_2(X)\rightarrow F_1(Y)$ is isomorphic to its image outside a $0$-dimensional subset of $F_2(X)$.

The following diagram is commutative: 
$$
\xymatrix{0 \ar[r] & T_{Z_H} \ar[r]\ar[d] &T_{\mathbb P_{F_2}}|_{Z_H}\ar[r]\ar[d] & N_{Z_H/\mathbb P_{F_2}} \ar[r]\ar[d] &0\\
  0 \ar[r] & (e_{F_2}^*T_{F_1(Y)})|_{Z_H}\ar[r] & (e_{F_2}^*T_{F_1(X)})|_{Z_H}\ar[r] & (e_{F_2}^*N_{F_1(Y)/F_1(X)})|_{Z_H}\ar[r] &0\rlap{.}}
$$
As $Z_H\subset \mathbb P_{F_2}$ is the zero locus of a regular section of $e_{F_2}^*\mathcal E_{2}|_{F_1(X)}$, we have $N_{Z_H/\mathbb P_{F_2}}\simeq (e_{F_2}^*\mathcal E_{2}|_{F_1(X)})|_{Z_H}$, so that the last vertical arrow in the diagram is an isomorphism. As the second vertical arrow is injective by Proposition \ref{prop_immersion_plan_lines}, the first is injective as well. So the snake lemma gives $(e_{F_2}^*T_{F_1(Y)})|_{Z_H}/T_{Z_H}\simeq N_{\mathbb P_{F_2}/F_1(X)}|_{Z_H}$.

According to \cite[Proposition 4]{Iliev-Manivel_cub_hyp_int_syst}, $\Im(j_H)$ is a (non-normal) Lagrangian surface of the hyper-K\"ahler manifold $F_1(Y)$. In particular, outside a codimension $2$ subset of $F_2(X)$, we have
$$
\Omega_{Z_H}\simeq \left(e_{F_2}^*T_{F_1(Y)}\right)_{Z_H}/T_{Z_H}.
$$

As both sheaves are locally free, the isomorphism holds globally; \textit{i.e.}, 
\begin{equation}\label{isom_cotang_normal2}\Omega_{F_2(X)}\simeq N_{\mathbb P_{F_2}/F_1(X)}|_{Z_H}.
\end{equation}

We can now prove Theorem \ref{thm_1}

\begin{proof}[Proof of Theorem \ref{thm_1}] Looking at (\ref{isom_cotang_normal2}) and (\ref{ex_seq_tgt_bundle_part1}), we see that we only have to check that $\mathcal O_e(1)|_{Z_H}\simeq \mathcal O_{Z_H}$.

For a (general) hyperplane $H\subset \mathbb P(V^*)$, we have a rational map 
  $\varphi\colon\Gr(3,V)\dashrightarrow \Gr(2,\langle H\rangle)$, $P\mapsto P\cap H$ whose indeterminacy locus is $\Gr(3,\langle H\rangle)$. The morphism $j_H\colon F_2(X)\simeq Z_H\rightarrow F_1(Y)$ is the restriction of the map $\varphi$ to $F_2(X)$. To get the result, we will show more generally that $\varphi^*\mathcal O_{\Gr(2,\langle H\rangle)}(-1)\otimes \mathcal O_{\Gr(3,V)}(1)$ restricts to the trivial line bundle on the open set where $\varphi$ is defined, \textit{i.e.}, on $\Gr(3,V)\backslash \Gr(3,\langle H\rangle)$.
  
The subvariety $\Gr(3,\langle H\rangle)\subset \Gr(3,V)$ is the zero locus of a regular section of $\mathcal E_3$, so that $N_{\Gr(3,\langle H\rangle)/\Gr(3,V)}\simeq \mathcal E_{3}|_{\Gr(3,\langle H\rangle)}$. After blowing up this locus, we get
$$
  \xymatrix{E_\tau \ar@{^{(}->}[r]^j\ar[d] & \widetilde{\Gr(3,V)}\ar[d]^\tau \ar[rd]^{\widetilde{\varphi}} &\\
    \Gr(3,\langle H\rangle)\ar@{^{(}->}[r]^i & \Gr(3,V)\ar@{-->}[r]^\varphi &\Gr(2,\langle H\rangle)\rlap{,}}
$$
where the exceptional divisor $E_\tau$ is isomorphic to $\mathbb P(\mathcal E_3^*)\simeq \mathbb P(\wedge^2\mathcal E_3\otimes \det(\mathcal E_3)^{-1})$. So $E_\tau$ is isomorphic to the flag variety $\Fl(2,3,\langle H\rangle)$, and $\widetilde\varphi\circ j$ correspond to the projection on the Grassmannian of lines; hence
$$
\mathcal O_{E_\tau}(1)\simeq j^*\widetilde\varphi^*\mathcal O_{\Gr(2,\langle H\rangle)}(1)\otimes \tau_{E_\tau}^*i^*\mathcal O_{\Gr(3,V)}(-1)\quad \text{in}\ \Pic\left(E_\tau\right).
$$
As the restriction $\Pic(\Gr(3,V))\rightarrow \Pic(\Gr(3,\langle H\rangle))$ is an isomorphism, so is $\Pic(\widetilde{\Gr(3,V)})\rightarrow \Pic(E_\tau)$; thus $$
\mathcal O_{\widetilde{\Gr(3,V)}}(-E)\simeq \widetilde\varphi^*\mathcal O_{\Gr(2,\langle H\rangle)}(1)\otimes \tau^*\mathcal O_{\Gr(3,V)}(-1)\quad \text{ in}\ \Pic\left(\widetilde{\Gr(3,V)}\right).
$$
Now pushing forward by $\tau$ the short exact sequence defining $E$, we get
$$
\tau_*\widetilde\varphi^*\mathcal O_{\Gr(2,\langle H\rangle)}(1)\otimes \mathcal O_{\Gr(3,V)}(-1)\simeq \tau_*\mathcal O_{\widetilde{\Gr(3,V)}}(-E)\simeq \mathcal I_{\Gr(3,\langle H\rangle)/ \Gr(3,V)},
$$
which is indeed trivial on $\Gr(3,V)\backslash \Gr(3,\langle H\rangle)$.
\end{proof}

\section{Gauss map of \texorpdfstring{$\boldsymbol{F_2(X)}$}{F\textunderscore 2 (X)}}
Let $X\subset \mathbb P(V^*)\simeq \mathbb P^6$ be a smooth cubic hypersurface such that $F_2(X)$ is a smooth (irreducible) surface. We begin this section with the following.  

\begin{theoreme}\label{thm_descrip_h_1_and_wedge} The following sequence is exact:
\begin{equation}\label{ex_seq_descrip_H_1}
0\longrightarrow H^1\left(\mathcal O_{F_2(X)}\right)\longrightarrow \Sym^2V\otimes \det(V)\xrightarrow{\varphi_{\eq_X}\otimes \id_{\det(V)}} V^*\otimes \det(V)\longrightarrow 0,
\end{equation}
where $\varphi_{\eq_X}$ is defined to be $e_i+e_j\mapsto \eq_X(e_i,e_j,\cdot)$.

Moreover, we have an inclusion $\bigwedge^2H^1(\mathcal O_{F_2(X)})\subset H^2(\mathcal O_{F_2(X)})$, which by Hodge symmetry yields $\bigwedge^2H^0(\Omega_{F_2(X)})\subset H^0(K_{F_2(X)})$.
\end{theoreme}

\begin{proof} As $\mathcal O_{F_2(X)}$ admits the Koszul resolution (\ref{ex_seq_koszul_resol}), to understand the cohomology groups $H^i(\mathcal O_{F_2(X)})$, we can use the spectral sequence
  $$
  E_1^{p,q}=H^q\left(G(3,V),\wedge^{-p}\Sym^3\mathcal E_3^*\right)\Longrightarrow H^{p+q}\left(\mathcal O_{F_2(X)}\right).
  $$

  As a reminder, we borrow from \cite{jiang_noether_lefschetz} (see also \cite{spandaw}) the following elementary presentation of the Borel--Weil--Bott theorem for a $G(3,W)$ with $\dim(W)=d$.
  
 For any vector space $L$ of dimension $f$ and any decreasing sequence of integers $a=(a_1,\dots,a_f)$, there is an irreducible $GL(L)$-representation (Weyl module) denoted by $\Gamma^{(a_1,\dots,a_f)}L$.

 To two decreasing sequences $a=(a_1,\dots,a_{d-e})$ and $b=(b_1,\dots,b_e)$, we can associate the sequence
 $$
 (\phi_1,\dots,\phi_d)=\phi(a,b):=(a_1-1,a_2-2,\dots,a_{d-e}-(d-e),b_1-(d-e+1),\dots,b_e-d).
 $$

 We measure how far $\phi(a,b)$ is from being decreasing by introducing $i(a,b):=\#\{\alpha<\beta,\ \phi_\alpha>\phi_\beta\}$.
 
 Finally, let us denote by $\phi(a,b)^+=(\phi_1^+,\dots,\phi_d^+)$ a re-ordering of $\phi(a,b)$ to make it non-increasing and set $\psi(a,b):=(\phi_1^++1,\dots,\phi_d^++d)$.
 
 The Borel--Weil--Bott theorem reads as follows. 

 \begin{theoreme}\label{thm_borel_weil_bott} We have
   \begin{enumerate}
\item  $H^q(G(3,W),\Gamma^a\mathcal Q_3^*\otimes\Gamma^b\mathcal E_3^*)=0$ for $q\neq i(a,b)$, 
\item $H^{i(a,b)}(G(3,W),\Gamma^a\mathcal Q_3^*\otimes\Gamma^b\mathcal E_3^*)=\Gamma^{\psi(a,b)}W$,
 \end{enumerate}
where $\mathcal Q_3$ and $\mathcal E_3$ are defined by \eqref{ex_seq_def_taut_3} and $\Gamma^{\psi(a,b)}W=0$ if $\psi(a,b)$ is not decreasing.
\end{theoreme} 

Now, we want to apply this theorem to compute the $E_1^{p,q}$ of the spectral sequence. Using Sage with the code
\begin{verbatim}
R=WeylCharacterRing("A2")
V=R(1,0,0)
for k in range(11): print k, V.symmetric_power(3).exterior_power(k)
\end{verbatim}
we get the decompositions into irreducible modules of $\wedge^k\Sym^3\mathcal E_3^*$. Then by the Borel--Weil--Bott theorem, we have

\allowdisplaybreaks
\begin{align*}
&(0) &\oplus_i^{12} H^i\left(\mathcal O_{G(3,V)}\right) &=\oplus_i H^i\left(\Gamma^{(0,\dots,0)}\mathcal Q_3^*\otimes\Gamma^{(0,0,0)}\mathcal E_3^*\right)\\
&  &  &=H^0\left(\mathcal O_{G(3,V)}\right)=\Gamma^{(0,\dots,0)}V\simeq \mathbb C,\\
&(1) &\oplus_i^{12} H^i\left(\Sym^3\mathcal E_3^*\right) &=\oplus_i^{12} H^i\left(\Gamma^{(3,0,0)}\mathcal E_3^*\right)=0,\\
&(2) &\oplus_iH^i\left(\wedge^2\Sym^3\mathcal E_3^*\right) &=\oplus_i H^i\left(\Gamma^{(3,3,0)}\mathcal E_3^*\oplus \Gamma^{(5,1,0)}\mathcal E_3^*\right)\\
&  &  &=H^4\left(\Gamma^{(5,1,0)}\mathcal E_3^*\right)=\Gamma^{(1,\dots,1,0)}V\simeq \wedge^6V,\\
&(3) &\oplus_i H^i\left(\wedge^3\Sym^3\mathcal E_3^*\right) &=\oplus_i H^i\left(\Gamma^{(3,3,3)}\mathcal E_3^*\oplus\Gamma^{(5,3,1)}\mathcal E_3^*\oplus \Gamma^{(6,3,0)}\mathcal E_3^*\oplus \Gamma^{(7,1,1)}\mathcal E_3^*\right)\\
&  &  &=H^4\left(\Gamma^{(7,1,1)}\mathcal E_3^*\right)=\Gamma^{(3,1,\dots,1)}V\simeq \Sym^2V\otimes \det(V),\\
&(4) &\oplus_i H^i\left(\wedge^4\Sym^3\mathcal E_3^*\right) &=\oplus_iH^i\left(\Gamma^{(6,3,3)}\mathcal E_3^*\oplus \Gamma^{(6,4,2)}\mathcal E_3^*\oplus\Gamma^{(6,6,0)}\mathcal E_3^*\oplus \Gamma^{(7,4,1)}\mathcal E_3^*\oplus\Gamma^{(8,3,1)}\mathcal E_3^*\right)\\
&  &  &=H^8\left(\Gamma^{(6,6,0)}\mathcal E_3^*\right)=\Gamma^{(2,\dots,2,0)}V\\
&  &  &\simeq \Sym^2V^*\otimes \det(V)^{\otimes 2},\\
&(5) &\oplus_iH^i\left(\wedge^5\Sym^3\mathcal E_3^*\right) &\simeq\oplus_iH^i\left(\Gamma^{(6,6,3)}\mathcal E_3^*\oplus\Gamma^{(7,4,4)}\mathcal E_3^*\oplus\Gamma^{(7,6,2)}\mathcal E_3^*\oplus\Gamma^{(8,4,3)}\mathcal E_3^*\oplus\Gamma^{(8,6,1)}\mathcal E_3^*\right.\\
&  &  &\ \ \ \ \ \ \ \left.\oplus\Gamma^{(9,4,2)}\mathcal E_3^*\right)\\
&  &  &=H^8\left(\Gamma^{(7,6,2)}\mathcal E_3^*\oplus\Gamma^{(8,6,1)}\mathcal E_3^*\right)\\
&  &  &=\Gamma^{(3,2,\dots,2)}V\oplus \Gamma^{(4,2\dots,2,1)}V\\
&  &  &\simeq \left(\Sym^2V\otimes V^*\right)\otimes \det(V)^{\otimes 2},\\
&(6) &\oplus_iH^i\left(\wedge^6\Sym^3\mathcal E_3^*\right) &\simeq \oplus_i H^i\left(\Gamma^{(7,7,4)}\mathcal E_3^*\oplus\Gamma^{(8,6,4)}\mathcal E_3^*\oplus\Gamma^{(9,6,3)}\mathcal E_3^*\oplus\Gamma^{(9,7,2)}\mathcal E_3^*\oplus\Gamma^{(10,4,4)}\mathcal E_3^*\right)\\
&  &  &=H^8\left(\Gamma^{(9,7,2)}\mathcal E_3^*\right)\\
&  &  &\simeq \Gamma^{(5,3,2\dots,2)}V\simeq \left(\wedge^2\Sym^2V\right)\otimes \det(V)^{\otimes 2},\\
&(7) &\oplus_iH^i\left(\wedge^7\Sym^3\mathcal E_3^*\right) &\simeq \oplus_iH^i\left(\Gamma^{(7,7,7)}\mathcal E_3^*\oplus\Gamma^{(9,7,5)}\mathcal E_3^*\oplus\Gamma^{(9,9,3)}\mathcal E_3^*\oplus\Gamma^{(10,7,4)}\mathcal E_3^*\right)\\
&  &  &=H^{12}\left(\Gamma^{(7,7,7)}\mathcal E_3^*\right)\simeq \Gamma^{(3,\dots,3)}V\simeq \det(V)^{\otimes 3},\\
&(8) &\oplus_iH^i\left(\wedge^8\Sym^3\mathcal E_3^*\right) &\simeq \oplus_iH^i\left(\Gamma^{(10,7,7)}\mathcal E_3^*\oplus \Gamma^{(10,9,5)}\mathcal E_3^*\right)\\
&  &  &=H^{12}\left(\Gamma^{(10,7,7)}\mathcal E_3^*\right)=\Gamma^{(6,3,\dots,3)}V\simeq \Sym^3V\otimes \det(V)^{\otimes 3},\\
&(9) &\oplus_iH^i\left(\wedge^9\Sym^3\mathcal E_3^*\right) &\simeq \oplus_iH^i\left(\Gamma^{(10,10,7)}\mathcal E_3^*\right)\\
&  &  &=H^{12}\left(\Gamma^{(10,10,7)}\mathcal E_3^*\right)\simeq \Gamma^{(6,6,3,\dots,3)}V,\\
&(10) &\oplus_iH^i\left(\wedge^{10}\Sym^3\mathcal E_3^*\right) &\simeq \oplus_iH^i\left(\Gamma^{(10,10,10)}\mathcal E_3^*\right)\\
&  &  &=H^{12}\left(\Gamma^{(10,10,10)}\mathcal E_3^*\right)\simeq \Gamma^{(6,6,6,3\dots,3)}V.
\end{align*}

To understand $H^1(\mathcal O_{F_2(X)})$, we have to examine the $E_\infty^{-i,i+1}$ for $i=0,\dots,10$. As $E_1^{-i,i+1}=0$ for any $i\neq 3$, we get $E_\infty^{-i,i+1}=0$ for $i \neq 3$.

On the other hand, for $r\geq 2$, $E_r^{-3,4}$ is defined as the (middle) cohomology of
$$
E_{r-1}^{-(2+r),2+r}\xrightarrow{d_{r-1}}E_{r-1}^{-3,4}\xrightarrow{d_{r-1}}E_{r-1}^{-4+r,6-r}.
$$ 
From the above computations, we see that $E_1^{-i,i}=0$ for $i\geq 3$, so that $E_r^{-i,i}=0$ for any $i\geq 3$ and $r\geq 1$.

So we get $E_2^{-3,4}=\Ker(d_1\colon E_1^{-3,4}\rightarrow E_1^{-2,4})$.

As $E_1^{-1,3}=0$, we have $E_2^{-1,3}=0$, so that $E_3^{-3,4}\simeq E_2^{-3,4}$.

As $E_1^{0,2}=0$, we have $E_3^{0,2}=0$, so that $E_4^{-3,4}\simeq E_2^{-3,4}$.

 As $E_1^{a,b}=0$ for any $a>0$, we get $E_\infty^{-3,4}\simeq E_2^{-3,4}$; \textit{i.e.}, the following sequence is exact:
 $$
 0\longrightarrow H^1\left(\mathcal O_{F_2(X)}\right)\longrightarrow E_1^{-3,4}\xrightarrow{d_1^{-3,4}}E_1^{-2,4}.
 $$

Now, $d_1^{-3,4}$ is given by contracting with the section defined by $\eq_X$, so that, choosing a basis $(e_0,\dots,e_6)$ of $V$, we have 
\begin{alignat*}{2}
d_1^{-3,4}\colon &\Sym^2V\otimes \det(V) &\ \longrightarrow \ &\wedge^6V\simeq V^*\otimes \det(V).\\
&(e_i+e_j)\otimes (e_0\wedge\cdots\wedge e_6) &\ \longmapsto \ &\sum_k\eq_X(e_i,e_j,e_k)\widehat{e_k}=\eq_X(e_i,e_j,\cdot)\otimes (e_0\wedge\cdots\wedge e_6)
\end{alignat*}

If this map is not surjective, we can choose the basis so that $e_0^*\otimes (e_0\wedge\cdots\wedge e_6)\notin \Im(d_1^{-3,4})$. Then we get $\eq_X(e_i,e_j,e_0)=0$ for any $i,j$, which means that the cubic hypersurface $X$ is a cone with vertex $[e_0]$.

So for a smooth cubic, $d_1^{-3,4}$ is surjective, so (\ref{ex_seq_descrip_H_1}) is exact.

Before tackling the case of $H^2(\mathcal O_{F_2(X)})$, we notice that the exterior square of (\ref{ex_seq_descrip_H_1}) gives the following exact sequence:
\begin{equation}\label{ex_seq_exterior_square_h_1}
    \begin{aligned}
0\longrightarrow \wedge^2H^1(\mathcal O_{F_2(X)})\longrightarrow (\wedge^2\Sym^2V)\otimes \det(V)^{\otimes 2}\xrightarrow{\varphi_{\eq_X}\otimes \id_{\Sym^2V\otimes \det(V)}}\Sym^2V\otimes V^*\otimes \det(V)^{\otimes 2}\\
\xrightarrow{\varphi_{\eq_X}\otimes \id_{V^*\otimes \det(V)}}\Sym^2V^*\otimes \det(V)^{\otimes 2}\longrightarrow 0.
    \end{aligned}
\end{equation}

To understand $H^2(\mathcal O_{F_2(X)})$, we have to examine the $E_\infty^{-i,i+2}$ for $i=0,\dots, 10$. As $E_1^{-i,i+2}=0$ for $i\neq 2,6,10$, we have $E_\infty^{-i,i+2}=0$ for $i\neq 2,6,10$.

\subsubsection*{Analysis of $\boldsymbol{E_\infty^{-2,4}}$} As $E_1^{-1,4}=0$, $E_2^{-2,4}$ is the cokernel of $d_1^{-3,4}$, which has just been proven to be surjective when $X$ is smooth. So $E_2^{-2,4}=0$, from which we get $E_\infty^{-2,4}=0$.

\subsubsection*{Analysis of $\boldsymbol{E_\infty^{-6,8}}$} Each $E_r^{-6,8}$ is the middle cohomology of
$$
E_{r-1}^{-(5+r),6+r}\xrightarrow{d_{r-1}}E_{r-1}^{-6,8}\xrightarrow{d_{r-1}}E_{r-1}^{-7+r,10-r}.
$$ 
From the above computations of the cohomology groups, we see that $E_1^{-(5+r),6+r}=0$ for any $r\geq 2$, so $E_{r-1}^{-(5+r),6+r}=0$ for any $r\geq 2$.

 So $E_2^{-6,8}=\Ker(d_1^{-6,8}\colon E_1^{-6,8}\rightarrow E_1^{-5,8})$.

  We see that $E_1^{-7+r,10-r}=0$ for any $r\geq 3$, so that $E_{r-1}^{-7+r,10-r}=0$ for any $r\geq 3$. As a result, we get $E_\infty^{-6,8}=E_2^{-6,8}$.

 From (\ref{ex_seq_exterior_square_h_1}), we get that $\Coker(d_1^{-6,8}\colon E_1^{-6,8}\rightarrow E_1^{-5,8})\simeq \Sym^2V^*\otimes \det(V)^{\otimes 2}$ and $E_\infty^{-6,8}=\Ker(d_1^{-6,8}\colon $ $ E_1^{-6,8}\rightarrow E_1^{-5,8})\simeq \wedge^2H^1(\mathcal O_{F_2(X)})$.

 Now, the spectral sequence computes the graded pieces of a filtration
 $$
 0=F^1\subset F^0\subset \cdots\subset F^{-10}\subset F^{-11}=H^2\left(\mathcal O_{F_2(X)}\right),
 $$
 and we have seen ($E_\infty^{-2,4}=0$) that all the graded pieces are trivial, but $\Gr_{-6}^F\simeq E_{\infty}^{-6,8}$ and (\textit{a priori}) $\Gr_{-10}^F\simeq E_\infty^{-10,12}$. As a result, we get $\wedge^2H^1(\mathcal O_{F_2(X)})\simeq E_\infty^{-6,8}=F^{-6}=\cdots=F^{-9}\subset F^{10}\subset H^2(\mathcal O_{F_2(X)})$, proving the inclusion.
\end{proof}

Moreover, we have the following proposition. 

\begin{proposition}\label{prop_descrip_h_0_omega} We have $H^0(\mathcal Q_{3}|_{F_2(X)}^*)\simeq H^0(\mathcal Q_3^*)\simeq V$ and $H^0(\Sym^2\mathcal E_{3}|_{F_2(X)})\simeq H^0(\Sym^2\mathcal E_3)\simeq \Sym^2V^*$, and the following sequence is exact: 
\begin{equation}\label{ex_seq_H_0_tgt_bundle_ex_seq}0\longrightarrow H^0\left(\mathcal Q_{3}^*|_{F_2(X)}\right)\longrightarrow H^0\left(\Sym^2\mathcal E_{3}|_{F_2(X)}\right)\longrightarrow H^0\left(\Omega_{F_2(X)}\right)\longrightarrow 0, 
\end{equation} where the first map is given by $v\mapsto \eq_X(v,\cdot,\cdot)$.
\end{proposition}

\begin{proof} To understand $H^0(\mathcal Q_{3}^*|_{F_2(X)})$, we use again the Koszul resolution (\ref{ex_seq_koszul_resol}) tensored by $\mathcal Q_3^*$. We have the spectral sequence
  $$
  E_1^{p,q}=H^q\left(G(3,V),\mathcal Q_3^*\otimes \wedge^{-p}\Sym^3\mathcal E_3^*\right)\Longrightarrow H^{p+q}\left(Q_{3}^*|_{F_2(X)}\right).
  $$
We  again use the Borel--Weil--Bott theorem \ref{thm_borel_weil_bott} to compute the cohomology groups on $G(3,V)$. The decompositions of the $\wedge^i\Sym\mathcal E_3^*$'s into irreducible modules have already been obtained in Theorem \ref{thm_descrip_h_1_and_wedge}. So we get
\allowdisplaybreaks
\begin{align*}
&  (0) &\oplus_iH^i\left(\mathcal Q_3^*\right) &\simeq \oplus_iH^i\left(\Gamma^{(1,0,0,0)}\mathcal Q_3^*\right)\\
&  &  &=H^0\left(\Gamma^{(1,0,0,0)}\mathcal Q_3^*\right)=V,\\
&(1) &\oplus_iH^i\left(\mathcal Q_3^*\otimes\Sym^3\mathcal E_3^*\right) &\simeq \oplus_iH^i\left(\Gamma^{(1,0,0,0)}\mathcal Q_3^*\otimes\Gamma^{(3,0,0)}\mathcal E_3^*\right)=0,\\
&(2) &\oplus_iH^i\left(\mathcal Q_3^*\otimes\wedge^2\Sym^3\mathcal E_3^*\right) &\simeq \oplus_iH^i\left(\Gamma^{(1,0,0,0)}\mathcal Q_3^*\otimes\left(\Gamma^{(3,3,0)}\mathcal E_3^*\oplus\Gamma^{(5,1,0)}\mathcal E_3^*\right)\right)=0,\\
&(3) &\oplus_iH^i\left(\mathcal Q_3^*\otimes\wedge^3\Sym^3\mathcal E_3^*\right) &\simeq \oplus_iH^i\left(\Gamma^{(1,0,0,0)}\mathcal Q_3^*\otimes\left(\Gamma^{(3,3,3)}\mathcal E_3^*\oplus\Gamma^{(5,3,1)}\mathcal E_3^*\oplus \Gamma^{(6,3,0)}\mathcal E_3^*\right.\right.\\
&  &  &\ \ \ \ \ \left.\left.\oplus\Gamma^{(7,1,1)}\mathcal E_3^*\right)\right)\\
&  &  &=H^4\left(\Gamma^{(1,0,0,0)}\mathcal Q_3^*\otimes\Gamma^{(7,1,1)}\mathcal E_3^*\right)\simeq \Gamma^{(3,2,1,\dots,1)}V,\\
&(4) &\oplus_iH^i\left(\mathcal Q_3^*\otimes\wedge^4\Sym^3\mathcal E_3^*\right) &\simeq \oplus_iH^i\left(\Gamma^{(1,0,0,0)}\mathcal Q_3^*\otimes\left(\Gamma^{(6,3,3)}\mathcal E_3^*\oplus\Gamma^{(6,4,2)}\mathcal E_3^*\oplus \Gamma^{(6,6,0)}\mathcal E_3^*\right.\right.\\
&  &  &\ \ \ \ \ \left.\left.\oplus\Gamma^{(7,4,1)}\mathcal E_3^*\oplus\Gamma^{(8,3,1)}\mathcal E_3^*\right)\right)\\
&  &  &=0, \\
&(5) &\oplus_iH^i\left(\mathcal Q_3^*\otimes\wedge^5\Sym^3\mathcal E_3^*\right) &\simeq \oplus_iH^i\left(\Gamma^{(1,0,0,0)}\mathcal Q_3^*\otimes\left(\Gamma^{(6,6,3)}\mathcal E_3^*\oplus\Gamma^{(7,4,4)}\mathcal E_3^*\oplus \Gamma^{(7,6,2)}\mathcal E_3^*\right.\right.\\
 & &  &\ \ \ \ \  \left.\left.\oplus\Gamma^{(8,4,3)}\mathcal E_3^*\oplus\Gamma^{(8,6,1)}\mathcal E_3^*\oplus \Gamma^{(9,4,2)}\mathcal E_3^*\right)\right)\\
&  &  &=0, \\
&(6) &\oplus_iH^i\left(\mathcal Q_3^*\otimes\wedge^6\Sym^3\mathcal E_3^*\right) &\simeq \oplus_iH^i\left(\Gamma^{(1,0,0,0)}\mathcal Q_3^*\otimes\left(\Gamma^{(7,7,4)}\mathcal E_3^*\oplus\Gamma^{(8,6,4)}\mathcal E_3^*\oplus \Gamma^{(9,6,3)}\mathcal E_3^*\right.\right.\\
 & &  &\ \ \ \ \ \left.\left.\oplus\Gamma^{(9,7,2)}\mathcal E_3^*\oplus\Gamma^{(10,4,4)}\mathcal E_3^*\right)\right)\\
 & &  &=H^8\left(\Gamma^{(1,0,0,0)}\mathcal Q_3^*\otimes \Gamma^{(9,7,2)}\mathcal E_3^*\right)\\
&  &  &\simeq\Gamma^{(5,3,3,2,\dots,2)}V,\\
&(7) &\oplus_iH^i\left(\mathcal Q_3^*\otimes\wedge^7\Sym^3\mathcal E_3^*\right) &\simeq \oplus_iH^i\left(\Gamma^{(1,0,0,0)}\mathcal Q_3^*\otimes\left(\Gamma^{(7,7,7)}\mathcal E_3^*\oplus\Gamma^{(9,7,5)}\mathcal E_3^*\oplus \Gamma^{(9,9,3)}\mathcal E_3^*\right.\right.\\
&  &  &\ \ \ \ \  \left.\left.\oplus\Gamma^{(10,7,4)}\mathcal E_3^*\right)\right)\\
&  &  &=0,\\
&(8) &\oplus_iH^i\left(\mathcal Q_3^*\otimes\wedge^8\Sym^3\mathcal E_3^*\right) &\simeq \oplus_iH^i\left(\Gamma^{(1,0,0,0)}\mathcal Q_3^*\otimes\left(\Gamma^{(10,7,7)}\mathcal E_3^*\oplus\Gamma^{(10,9,5)}\mathcal E_3^*\right)\right)\\
&  &  &=0,\\
&(9) &\oplus_iH^i\left(\mathcal Q_3^*\otimes\wedge^9\Sym^3\mathcal E_3^*\right) &\simeq \oplus_iH^i\left(\Gamma^{(1,0,0,0)}\mathcal Q_3^*\otimes \Gamma^{(10,10,7)}\mathcal E_3^*\right)=0,\\
&(10) &\oplus_i H^i\left(\mathcal Q_3^*\otimes\wedge^{10}\Sym^3\mathcal E_3^*\right) &\simeq \oplus_iH^i\left(\Gamma^{(1,0,0,0)}\mathcal Q_3^*\otimes \Gamma^{(10,10,10)}\mathcal E_3^*\right)\\
&  &  &=H^{12}\left(\Gamma^{(1,0,0,0)}\mathcal Q_3^*\otimes \Gamma^{(10,10,10)}\mathcal E_3^*\right)\simeq\Gamma^{(6,6,6,4,3,3,3)}V. 
\end{align*}

The graded pieces of the filtration on $H^0(\mathcal Q_{3}^*|_{F_2(X)})$ are given by $E_\infty^{-i,i}$, $i=0,\dots,10$. From the above calculations, we see that $E_1^{-i,i}=0$ for any $i\geq 1$; thus $E_\infty^{-i,i}=0$ for any $i\geq 1$.

On the other hand, $E_1^{0,0}=H^0(\mathcal Q_3^*)=V$, and as $E_r^{a,b}=0$ for any $a>0$, we have $E_r^{0,0}=\Coker(d_{r-1}\colon E_{r-1}^{-(r-1),r-2}  E_{r-1}^{0,0})$ for any $r\geq 2$. But the above calculations give $E_1^{-r,r-1}=0$ for $r\geq 0$, so that $E_r^{-r,r-1}=0$ for any $r\geq 1$. Thus $E_\infty^{0,0}=E_1^{0,0}$, proving that $H^0(\mathcal Q_{3}^*|_{F_2(X)})\simeq H^0(\mathcal Q_3^*)\simeq V$.

 Now, let us examine $H^0(\Sym^2\mathcal E_{3}|_{F_2(X)})$ using the spectral sequence 
 $$
 E_1^{p,q}=H^q\left(\Sym^2\mathcal E_3\otimes\wedge^{-p}\Sym^3\mathcal E_3^*\right)\Longrightarrow H^{p+q}\left(\Sym^2\mathcal E_{3}|_{F_2(X)}\right).
 $$
\indent Using Sage with the code 
\begin{Verbatim}[breaklines=true]
R=WeylCharacterRing("A2")
V=R(1,0,0)
W=R(0,0,-1)
for k in range(11): print k, W.symmetric_power(2)*V.symmetric_power(3).exterior_power(k)
\end{Verbatim} 
and the Borel--Weil--Bott theorem \ref{thm_borel_weil_bott}, we get 
\allowdisplaybreaks
\begin{align*}
&(0) &\oplus_i H^i\left(\Sym^2\mathcal E_3\right) &\simeq \oplus_iH^i\left(\Gamma^{(0,0,-2)}\mathcal E_3^*\right)\\
&  &  &=H^0\left(\Gamma^{(0,0,-2)}\mathcal E_3^*\right)\simeq \Gamma^{(0,\dots,0,-2)}V\simeq \Sym^2V^*,\\
&(1) &\oplus_iH^i\left(\Sym^2\mathcal E_3\otimes\Sym^3\mathcal E_3^*\right) &\simeq \oplus_iH^i\left(\Gamma^{(1,0,0)}\mathcal E_3^*\oplus \Gamma^{(2,0,-1)}\mathcal E_3^*\oplus \Gamma^{(3,0,-2)}\mathcal E_3^*\right)=0,\\
&(2) &\oplus_iH^i\left(\Sym^2\mathcal E_3\otimes\wedge^2\Sym^3\mathcal E_3^*\right) &\simeq \oplus_i H^i\left(\left(\Gamma^{(3,1,0)}\mathcal E_3^*\right)^{\oplus 2}\oplus \Gamma^{(3,2,-1)}\mathcal E_3^*\oplus \Gamma^{(3,3,-2)}\mathcal E_3^*\right.\\
&  &  &\ \ \ \ \ \left.\oplus\Gamma^{(4,0,0)}\mathcal E_3^*\oplus \Gamma^{(4,1,-1)}\mathcal E_3^*\oplus \Gamma^{(5,1,-2)}\mathcal E_3^*\oplus \Gamma^{(5,0,-1)}\mathcal E_3^*\right)\\
&  &  &=H^4\left(\Gamma^{(5,1,-2)}\mathcal E_3^*\oplus \Gamma^{(5,0,-1)}\mathcal E_3^*\right)\\
&  &  &\simeq \Gamma^{(1,\dots,1,-2)}V\oplus \Gamma^{(1,\dots,1,0,-1)}V,\\
&(3) &\oplus_i H^i\left(\Sym^2\mathcal E_3\otimes\wedge^3\Sym^3\mathcal E_3^*\right) &\simeq \oplus_iH^i\left(\left(\Gamma^{(3,3,1)}\mathcal E_3^*\right)^{\oplus 2}\oplus \Gamma^{(4,2,1)}\mathcal E_3^*\oplus \left(\Gamma^{(4,3,0)}\mathcal E_3^*\right)^{\oplus 2}\right.\\
&  &  &\ \ \ \ \ \oplus \left(\Gamma^{(5,1,1)}\mathcal E_3^*\right)^{\oplus 2}\oplus \left(\Gamma^{(5,2,0)}\mathcal E_3^*\right)^{\oplus 2}\oplus \left(\Gamma^{(5,3,-1)}\mathcal E_3^*\right)^{\oplus 2}\\
&  &  &\ \ \ \ \ \oplus \left(\Gamma^{(6,1,0)}\mathcal E_3^*\right)^{\oplus 2}\oplus \Gamma^{(6,2,-1)}\mathcal E_3^*\oplus \Gamma^{(6,3,-2)}\mathcal E_3^*\\
&  &  &\ \ \ \ \ \left.\oplus \Gamma^{(7,1,-1)}\mathcal E_3^*\right)\\
&  &  &=H^4\left(\left(\Gamma^{(5,1,1)}\mathcal E_3^*\right)^{\oplus 2}\oplus\left(\Gamma^{(6,1,0)}\mathcal E_3^*\right)^{\oplus 2}\oplus \Gamma^{(7,1,-1)}\mathcal E_3^*\right)\\
&  &  &\simeq \det(V)^{\oplus 2}\oplus \left(\Gamma^{(2,1,\dots,1,0)}V\right)^{\oplus 2}\oplus \Gamma^{(3,1,\dots,1,-1)}V,\\
&(4) &\oplus_i H^i\left(\Sym^2\mathcal E_3\otimes\wedge^4\Sym^3\mathcal E_3^*\right) &\simeq \oplus_iH^i\left(\Gamma^{(4,3,3)}\mathcal E_3^*\oplus \Gamma^{(4,4,2)}\mathcal E_3^*\oplus \left(\Gamma^{(5,3,2)}\mathcal E_3^*\right)^{\oplus 2}\right.\\
&  &  &\ \ \ \ \ \oplus \left(\Gamma^{(5,4,1)}\mathcal E_3^*\right)^{\oplus 2}\oplus \Gamma^{(6,2,2)}\mathcal E_3^*\oplus \left(\Gamma^{(6,3,1)}\mathcal E_3^*\right)^{\oplus 4}\\
&  &  &\ \ \ \ \ \oplus \left(\Gamma^{(6,4,0)}\mathcal E_3^*\right)^{\oplus 3}\oplus \Gamma^{(6,5,-1)}\mathcal E_3^*\oplus \Gamma^{(6,6,-2)}\mathcal E_3^*\\
&  &  &\ \ \ \ \ \oplus \left(\Gamma^{(7,2,1)}\mathcal E_3^*\right)^{\oplus 2}\oplus \left(\Gamma^{(7,3,0)}\mathcal E_3^*\right)^{\oplus 2}\oplus \Gamma^{(7,4,-1)}\mathcal E_3^*\\
&  &  &\ \ \ \ \ \left.\oplus \Gamma^{(8,1,1)}\mathcal E_3^*\oplus \Gamma^{(8,2,0)}\mathcal E_3^*\oplus \Gamma^{(8,3,-1)}\mathcal E_3^*\right)\\
&  &  &=\underbrace{H^4\left(\Gamma^{(8,1,1)}\mathcal E_3^*\right)}_{\simeq\, \Sym^3V\otimes \det(V)}\oplus \underbrace{H^8\left(\Gamma^{(6,6,-2)}\mathcal E_3^*\right)}_{\simeq\, \Gamma^{(2,\dots,2,-2)}V},\\
&(5) &\oplus_i H^i\left(\Sym^2\mathcal E_3\otimes\wedge^5\Sym^3\mathcal E_3^*\right) &\simeq\oplus_iH^i\left(\Gamma^{(5,4,4)}\mathcal E_3^*\oplus \left(\Gamma^{(6,4,3)}\mathcal E_3^*\right)^{\oplus 3}\oplus \left(\Gamma^{(6,5,2)}\mathcal E_3^*\right)^{\oplus 2}\right.\\
&  &  &\ \ \ \ \ \oplus \left(\Gamma^{(6,6,1)}\mathcal E_3^*\right)^{\oplus 3}\oplus \Gamma^{(7,3,3)}\mathcal E_3^*\oplus \left(\Gamma^{(7,4,2)}\mathcal E_3^*\right)^{\oplus 4}\\
&  &  &\ \ \ \ \ \oplus \left(\Gamma^{(7,5,1)}\mathcal E_3^*\right)^{\oplus 2}\oplus \left(\Gamma^{(7,6,0)}\mathcal E_3^*\right)^{\oplus 2}\oplus \left(\Gamma^{(8,3,2)}\mathcal E_3^*\right)^{\oplus 2}\\
&  &  &\ \ \ \ \ \oplus \left(\Gamma^{(8,4,1)}\mathcal E_3^*\right)^{\oplus 3}\oplus \Gamma^{(8,5,0)}\mathcal E_3^*\oplus \Gamma^{(8,6,-1)}\mathcal E_3^*\\
&  &  &\ \ \ \ \ \left. \oplus \Gamma^{(9,2,2)}\mathcal E_3^*\oplus \Gamma^{(9,3,1)}\mathcal E_3^*\oplus \Gamma^{(9,4,0)}\mathcal E_3^*\right)\\
&  &  &=H^8\left(\left(\Gamma^{(6,6,1)}\mathcal E_3^*\right)^{\oplus 3}\oplus \left(\Gamma^{(7,6,0)}\mathcal E_3^*\right)^{\oplus 2}\oplus \Gamma^{(8,6,-1)}\mathcal E_3^*\right)\\
&  &  &\simeq \left(\Gamma^{(2,\dots,2,1)}V\right)^{\oplus 3}\oplus \left(\Gamma^{(3,2,\dots,2,0)}V\right)^{\oplus 2}\oplus \Gamma^{(4,2,\dots,2,-1)}V,\\
&(6) &\oplus_i H^i\left(\Sym^2\mathcal E_3\otimes\wedge^6\Sym^3\mathcal E_3^*\right) &\simeq \oplus_i H^i\left(\Gamma^{(6,6,4)}\mathcal E_3^*\oplus \left(\Gamma^{(7,5,4)}\mathcal E_3^*\right)^{\oplus 2}\oplus \left(\Gamma^{(7,6,3)}\mathcal E_3^*\right)^{\oplus 3}\right. \\
&  &  &\ \ \ \ \ \oplus \left(\Gamma^{(7,7,2)}\mathcal E_3^*\right)^{\oplus 2}\oplus \left(\Gamma^{(8,4,4)}\mathcal E_3^*\right)^{\oplus 2}\oplus \left(\Gamma^{(8,5,3)}\mathcal E_3^*\right)^{\oplus 2}\\
&  &  &\ \ \ \ \ \oplus \left(\Gamma^{(8,6,2)}\mathcal E_3^*\right)^{\oplus 3}\oplus \Gamma^{(8,7,1)}\mathcal E_3^*\oplus \left(\Gamma^{(9,4,3)}\mathcal E_3^*\right)^{\oplus 2}\\
&  &  &\ \ \ \ \ \oplus \left(\Gamma^{(9,5,2)}\mathcal E_3^*\right)^{\oplus 2}\oplus \left(\Gamma^{(9,6,1)}\mathcal E_3^*\right)^{\oplus 2}\oplus \Gamma^{(9,7,0)}\mathcal E_3^*\\
  &  &  &\ \ \ \ \ \left. \oplus \Gamma^{(10,4,2)}\mathcal E_3^*\right)\\
&  &  &=H^8\left(\left(\Gamma^{(7,7,2)}\mathcal E_3^*\right)^{\oplus 2}\oplus \left(\Gamma^{(8,6,2)}\mathcal E_3^*\right)^{\oplus 3}\oplus \Gamma^{(8,7,1)}\mathcal E_3^*\right.\\
&  &  &\ \ \ \ \ \left. \oplus \left(\Gamma^{(9,6,1)}\mathcal E_3^*\right)^{\oplus 2}\oplus \Gamma^{(9,7,0)}\mathcal E_3^*\right)\\
&  &  &\simeq \left(\Gamma^{(3,3,2,\dots, 2)}V\right)^{\oplus 2}\oplus \left(\Gamma^{(4,2,\dots,2)}V\right)^{\oplus 3}\oplus \Gamma^{(4,3,2,\dots,2,1)}V\\*
&  &  &\ \ \ \ \ \oplus \left(\Gamma^{(5,2,\dots,2,1)}V\right)^{\oplus 2}\oplus \Gamma^{(5,3,2\dots,2,0)}V,\\
&(7) &\oplus_i H^i\left(\Sym^2\mathcal E_3\otimes\wedge^7\Sym^3\mathcal E_3^*\right) &\simeq \oplus_i H^i\left(\left(\Gamma^{(7,7,5)}\mathcal E_3^*\right)^{\oplus 2}\oplus \Gamma^{(8,6,5)}\mathcal E_3^*\oplus \left(\Gamma^{(8,7,4)}\mathcal E_3^*\right)^{\oplus 2}\right. \\
&  &  &\ \ \ \ \ \oplus \Gamma^{(9,5,5)}\mathcal E_3^*\oplus \left(\Gamma^{(9,6,4)}\mathcal E_3^*\right)^{\oplus 2}\oplus \left(\Gamma^{(9,7,3)}\mathcal E_3^*\right)^{\oplus 3}\\
&  &  &\ \ \ \ \ \oplus \Gamma^{(9,8,2)}\mathcal E_3^*\oplus \Gamma^{(9,9,1)}\mathcal E_3^*\oplus \Gamma^{(10,5,4)}\mathcal E_3^*\\
&  &  &\ \ \ \ \ \left. \oplus \Gamma^{(10,6,3)}\mathcal E_3^*\oplus \Gamma^{(10,7,2)}\mathcal E_3^*\right)\\
&  &  &=H^8\left(\Gamma^{(9,8,2)}\mathcal E_3^*\oplus \Gamma^{(9,9,1)}\mathcal E_3^*\oplus \Gamma^{(10,7,2)}\mathcal E_3^*\right)\\
&  &  &\simeq \Gamma^{(5,4,2\dots,2)}V\oplus \Gamma^{(5,5,2,\dots,2,1)}V\oplus \Gamma^{(6,3,2,\dots,2)}V,\\
&(8) &\oplus_i H^i\left(\Sym^2\mathcal E_3\otimes\wedge^8\Sym^3\mathcal E_3^*\right) &\simeq \oplus_i H^i\left(\Gamma^{(8,7,7)}\mathcal E_3^*\oplus \Gamma^{(9,7,6)}\mathcal E_3^*\oplus \Gamma^{(9,8,5)}\mathcal E_3^*\right. \\
&  &  &\ \ \ \ \ \oplus \Gamma^{(9,9,4)}\mathcal E_3^*\oplus \left(\Gamma^{(10,7,5)}\mathcal E_3^*\right)^{\oplus 2}\oplus \Gamma^{(10,8,4)}\mathcal E_3^*\\
&  &  &\ \ \ \ \ \left. \oplus \Gamma^{(10,9,3)}\mathcal E_3^*\right)\\
&  &  &=H^{12}\left(\Gamma^{(8,7,7)}\mathcal E_3^*\right)\simeq \Gamma^{(4,3,\dots,3)}V,\\
&(9) &\oplus_i H^i\left(\Sym^2\mathcal E_3\otimes\wedge^9\Sym^3\mathcal E_3^*\right) &\simeq \oplus_i H^i\left(\Gamma^{(10,8,7)}\mathcal E_3^*\oplus \Gamma^{(10,9,6)}\mathcal E_3^*\oplus \Gamma^{(10,10,5)}\mathcal E_3^*\right)\\
&  &  &=H^{12}\left(\Gamma^{(10,8,7)}\mathcal E_3^*\right)\simeq \Gamma^{(6,4,3,\dots,3)}V,\\
&(10) &\oplus_i H^i\left(\Sym^2\mathcal E_3\otimes\wedge^{10}\Sym^3\mathcal E_3^*\right) &\simeq \oplus_i H^i\left(\Gamma^{(10,10,8)}\mathcal E_3^*\right)\\
&  &  &=H^{12}\left(\Gamma^{(10,10,8)}\mathcal E_3^*\right)\simeq \Gamma^{(6,6,4,3,\dots,3)}V. 
\end{align*}

The graded pieces of the filtration on $H^0(\Sym^2\mathcal E_{3}|_{F_2(X)})$ are given by the $E_\infty^{-i,i}$. We have $E_\infty^{-i,i}=0$ for any $i\neq 0,4$ since $E_1^{-i,i}=0$ for $i\neq 0,4$.

As $E_r^{a,b}=0$ for any $a>0$ and $E_r^{-r,r-1}=0$ (because $E_1^{-r,r-1}=0$) for any $r\geq 1$, we have $E_\infty^{0,0}=E_1^{0,0}$.

In particular, $H^0(\Sym^2\mathcal E_3)\simeq E_\infty^{0,0}\subset H^0(\Sym^2\mathcal E_{3}|_{F_2(X)})$. As $h^0(\Sym^2\mathcal E_3)=\dim(\Sym^2V^*)=28$, we have $h^0(\Sym^2\mathcal E_{3}|_{F_2(X)})\geq 28$. By Hodge symmetry, $h^0(\Omega_{F_2(X)})=h^1(\mathcal O_{F_2(X)})=21$ (see Theorem \ref{thm_descrip_h_1_and_wedge}). So the exactness of the sequence
$$
0\longrightarrow H^0\left(\mathcal Q_{3}^* |_{F_2(X)}\right)\longrightarrow H^0\left(\Sym^2\mathcal E_{3}|_{F_2(X)}\right)\longrightarrow H^0\left(\Omega_{F_2(X)}\right)
$$
implies $H^0(\Sym^2\mathcal E_3)=H^0(\Sym^2\mathcal E_{3}|_{F_2(X)})$ and the surjectivity of the last map.
\end{proof}

According to Theorem \ref{thm_descrip_h_1_and_wedge}, $\bigwedge^2H^0(\Omega_{F_2(X)})\subset H^0(K_{F_2(X)})$. As $K_{F_2(X)}\simeq \mathcal O_{G(3,V)}(3)|_{F_2(X)}$, the map $\rho\colon F_2(X)\dashrightarrow |\bigwedge^2H^0(\Omega_{F_2(X)})|$ is the composition of the degree $3$ Veronese of the natural embedding $F_2(X)\subset G(3,V)$ followed by a linear projection. Moreover, we have the following. 

\begin{lemme}\label{lem_bpf_wedge_kernel_alb}\leavevmode
  \begin{enumerate}
    \item\label{lbwka-1} The canonical bundle $K_{F_2(X)}$ is generated by the sections in $\bigwedge^2H^0(\Omega_{F_2(X)})\subset H^0(K_{F_2(X)})$. In particular, $|\bigwedge^2H^0(\Omega_{F_2(X)})|$ is base-point-free.

\item\label{lbwka-2}  For any $[P]\in F_2(X)$, the following sequence is exact:
  $$
  0\longrightarrow \mathcal K_{[P]}\longrightarrow H^0\left(\Omega_{F_2(X)}\right)\xrightarrow{\ev([P])} \Omega_{F_2(X),[P]}\longrightarrow 0,
  $$
  where $\mathcal K_{[P]}=\{Q\in H^0(\mathcal O_{\mathbb P^6}(2)),\ P\subset \{Q=0\}\}/\Span( (\eq_X(x,\cdot,\cdot))_{x\in \langle P\rangle})$. 
\end{enumerate}
\end{lemme}

\begin{proof}\eqref{lbwka-1}  As $\mathcal E_{3}|_{F_2(X)}$ is globally generated (as a restriction of $\mathcal E_3$, which is globally generated, by (\ref{ex_seq_def_taut_3})), $\Sym^2\mathcal E_{3}|_{F_2(X)}$ is also globally generated. The same holds for $\mathcal Q_{3}^*|_{F_2(X)}$ (by (\ref{ex_seq_def_taut_3})). So applying the evaluation to (\ref{ex_seq_H_0_tgt_bundle_ex_seq}), we get the commutative diagram
  $$
  \xymatrix@-1pc{0\ar[r] &H^0\left(\mathcal Q_{3}|_{F_2(X)}^*\right)\otimes \mathcal O_{F_2(X)}\ar[r]\ar[d]^{\ev_1} &H^0\left(\Sym^2\mathcal E_{3}|_{F_2(X)}\right)\otimes \mathcal O_{F_2(X)} \ar[r]\ar[d]^{\ev_2} &H^0\left(\Omega_{F_2(X)}\right)\otimes \mathcal O_{F_2(X)} \ar[r]\ar[d]^{\ev_3} &0\\
    0\ar[r] &\mathcal Q_{3}^*|_{F_2(X)}\ar[r] &\Sym^2\mathcal E_{3}|_{F_2(X)}\ar[r] &\Omega_{F_2(X)}\ar[r] &0,}
  $$
  where the bottom row is (\ref{ex_seq_tgt_bundle_seq}). As $\ev_2$ is surjective, we get that $\ev_3$ is also surjective; \textit{i.e.}, $\Omega_{F_2(X)}$ is globally generated. Then taking the exterior square of $\ev_3$, we get that $\wedge^2\ev_3$ is surjective:
  $$
  \bigwedge^2H^0\left(\Omega_{F_2(X)}\right)\otimes \mathcal O_{F_2(X)}\xrightarrowdbl{\wedge^2 \ev_3}\wedge^2\Omega_{F_2(X)}.
  $$
Now a base point of $|\bigwedge^2H^0(\Omega_{F_2(X)})|$ would be a point where $\wedge^2 \ev_3$ fails to be surjective. So $|\bigwedge^2H^0(\Omega_{F_2(X)})|$ is base-point-free.

\eqref{lbwka-2} As $H^0(\mathcal Q_{3}^*|_{F_2(X)})\simeq H^0(\mathcal Q_3^*)\simeq V$ by Proposition \ref{prop_descrip_h_0_omega}, (\ref{ex_seq_def_taut_3}) yields $\ker(\ev_1)\simeq \mathcal E_{3}^*|_{F_2(X)}$, so the snake lemma gives the exact sequence.
\end{proof}  

Now, let us come back to the Gauss map of $F_2(X)$, that we have defined to be 
\begin{alignat*}{2}
\mathcal G\colon & \alb_{F_2}(F_2(X)) &\ \longdashrightarrow \ & G\left(2, T_{\Alb(F_2(X)),0}\right),\\
& t  & \ \longmapsto \ & T_{\alb_{F_2}(F_2(X))-t,0}
\end{alignat*}
where $\alb_{F_2}(F_2(X))-t$ is the translation of $\alb_{F_2}(F_2(X))\subset \Alb(F_2(X))$ by $-t\in \Alb(F_2(X))$. It is defined on the smooth locus of $\alb_{F_2}(F_2(X))$.

According to \cite[Section (III)]{Coll_cub}, $T {\alb_{F_2}}$
is injective. So the indeterminacies of $\mathcal G$ are resolved by the pre-composition with $\alb_{F_2}$, \textit{i.e.},  
\begin{align*}
F_2(X) & \longrightarrow G\left(2,T_{\Alb(F_2(X)),0}\right)\\
t & \longmapsto T_{-\alb_{F_2}(t)}\text{Translate}(-\alb_{F_2}(t))\left(T_t \alb_{F_2}\left(T_{F_2(X),t}\right)\right).
\end{align*}

\indent We have the Pl\"ucker embedding 
$$
G\left(2,T_{\Alb(F_2(X)),0}\right)\simeq G\left(2,H^0\left(\Omega_{F_2(X)}\right)^*\right)\subset \mathbb P\left(\bigwedge^2H^0\left(\Omega_{F_2(X)}\right)^*\right)
$$
and the commutative diagram
$$
\xymatrix{F_2(X)\ar[r]^{\alb_{F_2}}\ar[dd]^{\rho} &\alb_{F_2}(F_2(X))\ar@{.>}[d]^{\mathcal G}\\
 &G\left(2,H^0(\Omega_{F_2(X)})^*\right)\ar@{^{(}->}[d]\\
  \left|\bigwedge^2H^0\left(\Omega_{F_2(X)}\right)\right|\ar[r]^*[@]{\cong\;\;} &\mathbb P\left(\bigwedge^2H^0\left(\Omega_{F_2(X)}\right)^*\right)\rlap{.}}
$$

The following proposition completes the proof of Theorem \ref{thm_gauss_map}. 

\begin{proposition}\label{prop_rho_embedding} The morphism $\rho$ is an embedding, which implies that $\alb_{F_2}$ is an isomorphism unto its image and $\mathcal G$ is an embedding.
\end{proposition}

\begin{proof} Let us denote by $J_X$ the Jacobian ideal of $X$, \textit{i.e.}, the ideal of the polynomial ring generated by
  $\left(\frac{\partial \eq_X}{\partial X_i}\right)_{i=0,\dots,6}$ and by $J_{X,2}$ its homogeneous part of degree $2$. By Proposition \ref{prop_result_Collino_smoothness}, for any $[P]\in F_2(X)$, $\dim(J_{X,2}|_{P})=4$, so that $\dim(J_X\cap \{Q\in H^0(\mathcal O_{\mathbb P^6}(2)),\ P\subset \{Q=0\}\})=3$. We have the following. 

  \begin{lemme}\label{lem_jacobian_ideal}\leavevmode
    \begin{enumerate}
      \item\label{lji-1} For $[P]\in G(3,V)$, the codimension of\, $L_P^2:=\{Q\in H^0(\mathcal O_{\mathbb P^6}(2)),\ P\subset \{Q=0\}\}$ in $H^0(\mathcal O_{\mathbb P^6}(2))$ is $6$. For $[P]\neq [P']\in G(3,V)$, the codimension of\, $L_{P,P'}^2:=\{Q\in H^0(\mathcal O_{\mathbb P^6}(2)),\ P,P'\subset \{Q=0\}\}$ in $L_P^2$ is 
\begin{enumerate}
\item $6$ if $P\cap P'=\emptyset$,
\item  $5$ if $P\cap P'=\{\pt\}$,
\item $3$ if $P\cap P'=\{\Line\}$.
\end{enumerate}
\item\label{lji-2} For $[P]\neq [P']\in F_2(X)$ such that $P\cap P'=\{\Line\}$, we have $\dim(J_X\cap L_{P,P'}^2)\geq 1$, and if\, $X$ is general, we even have $\dim(J_X\cap L_{P,P'}^2)\geq 2$. So  $L_P^2/(J_X\cap L_P^2)+ L_{P,P'}^2\subsetneq L_P^2$, and for $X$ general, $\dim(L_P^2/(J_X\cap L_P^2)+ L_{P,P'}^2)\geq 2$.
\end{enumerate}
    \end{lemme}

  \begin{proof} \eqref{lji-1} This follows from a direct calculation.

\eqref{lji-2} Up to a projective transformation, we can assume $P=\{X_0=\cdots=X_3=0\}$ and $P'=\{X_0=X_1=X_2=X_4=0\}$. Then $\eq_X$ is of the form (\ref{normal_form_0}) with the additional conditions $Q_3(0,X_5,X_6)=0$, $D_5(0,0,0,X_3)=0$, $D_6(0,0,0,X_3)=0$, $R(0,0,0,X_3)=0$.

By definition, the quadrics of the Jacobian ideal are $\frac{\partial \eq_X}{\partial X_i}$, and according to Proposition \ref{prop_result_Collino_smoothness}, $\left(\frac{\partial \eq_X}{\partial X_i}|_{P}\right)_{i=0,\dots,3}$ are linearly independent, so that
$$J_X\cap L_P^2=\Span\left(\left(\frac{\partial \eq_X}{\partial X_i}|_{P}\right)_{i=4,5,6}\right).
$$
For $i\in\{4,5,6\}$,
$$
\frac{\partial \eq_X}{\partial X_i}=X_0\frac{\partial Q_0}{\partial X_i}+X_1\frac{\partial Q_1}{\partial X_i}+X_2\frac{\partial Q_2}{\partial X_i}+X_3\frac{\partial Q_3}{\partial X_i}+D_i
  $$
  which, when restricted to $P'$, gives $\frac{\partial \eq_X}{\partial X_i}|_{P'}=X_3\frac{\partial Q_3}{\partial X_i}(0,X_5,X_6)+D_i(0,0,0,X_3)$. But since $Q_3(0,X_5,X_6)=0$, we have $\frac{\partial Q_3}{\partial X_i}(0,X_5,X_6)=0$ for $i=5,6$, so that $\frac{\partial \eq_X}{\partial X_5}|_{P'}=0=\frac{\partial \eq_X}{\partial X_6}|_{P'}$, \textit{i.e.}, $\frac{\partial \eq_X}{\partial X_5}$, $\frac{\partial \eq_X}{\partial X_6}\in L_{P,P'}^2\cap J_X$. For $X$ general, those two quadric polynomials are independent.
  
 We have $\dim(J_X\cap L^2_P + L^2_{P,P'})=\dim(J_X\cap L_P^2)+\dim(L^2_{P,P'})-\dim(J_X\cap L_{P,P'}^2)$, which, by the first item of the lemma, yields the result.
\end{proof}

  According to Lemma~\ref{lem_jacobian_ideal}, for $[P]\neq [P']\in F_2(X)$, we can always find a quadric $Q\in H^0(\mathcal O_{\mathbb P^6}(2))$ such that $0\neq \overline Q\in L_P^2/(J_X\cap L_P^2+L_{P,P'}^2)$; in particular, $Q|_{P}=0$
  but $Q|_{P'}\neq 0$. Pick another $Q'\in H^0(\mathcal O_{\mathbb  P^6}(2))\backslash (L_P^2\cup L_{P'}^2)$ (\textit{i.e.}, $Q'|_{P}\neq 0$, $Q'|_{P'}\neq 0$) such that $Q'|_{P'}$ is independent of $Q|_{P'}$ and $Q$ and $Q'$ are independent modulo $J_{X,2}$ ($\dim(H^0(\mathcal O_{\mathbb P^6}(2))/(J_{X,2}\oplus\mathbb C[Q]))=5$).

 By Proposition \ref{prop_descrip_h_0_omega}, such quadrics give rise to $1$-forms on $F_2(X)$. Then $Q\wedge Q'\in \bigwedge^2H^0(\Omega_{F_2(X)})$ vanishes at $[P]$ but not at $[P']$; \textit{i.e.}, $|\bigwedge^2H^0(\Omega_{F_2(X)})|$ separates points.

 Now, given a $[P]\in F_2(X)$, we recall that $$T_{[P]}F_2(X)=\{u\in \Hom(\langle P\rangle,V/\langle P\rangle),\ \eq_X(x,x,u(x))=0\ \forall x\in \langle P\rangle\}$$ (the first order of $\eq_X(x+u(x),x+u(x),x+u(x))$ is $0$ for all $x\in \langle P\rangle$).

 Let $Q\in L_P^2$ be such that $0\neq \overline Q\in H^0(\mathcal O_{\mathbb P^6}(2))/J_{X,2}$ and $T_{[P]}F_2(Q)\cap T_{[P]}F_2(X)=\{0\}$. Pick a non-zero $\overline Q'\in H^0(\mathcal O_{\mathbb P^6}(2))/J_{X,2}$ such that $Q'|_{P}\neq 0$; then $Q\wedge Q'\in \bigwedge^2H^0(\Omega_{F_2(X)})$ and $(Q\wedge Q')|_{P}=0$.
 
 Moreover, given a $u\in T_{[P]}F_2(X)$, we have $d_{[P]}Q(u)\wedge Q'|_{P}+Q|_{P}\wedge d_{[P]}Q'(u)=d_{[P]}Q(u)\wedge Q'|_{P}$, where $d_{[P]}Q(u)$ is the quadratic form $x\mapsto \eq_Q(x,u(x))$ and is non-trivial since $T_{[P]}F_2(Q)\cap T_{[P]}F_2(X)=\{0\}$. Then for $Q$ generic (containing $P$ and such that $T_{[P]}F_2(Q)\cap T_{[P]}F_2(X)=\{0\}$), $d_{[P]}Q(u)$ is linearly independent of $Q'|_{P}$, so that $Q\wedge Q'$ does vanish along the tangent vector $u$. So $|\bigwedge^2H^0(\Omega_{F_2(X)})|$ separates tangent directions.
\end{proof}

\section{Variety of osculating planes of a cubic 4-fold}
In (\ref{def_var_of_oscul_planes}), we have previously introduced, for a smooth cubic $4$-fold containing no plane $Z\subset \mathbb P(H^*)\simeq \mathbb P^5$, the variety of osculating planes 
$F_0(Z):=\{[P]\in G(3,H),\ \exists \ell\subset P\ \text{line s.t.}\ P\cap Z=\ell\ \text{(set-theoretically)}\}$.

The variety $F_0(Z)$ lives naturally in $\Fl(2,3,H)$, \textit{i.e.},
$$
F_0(Z)=\{([\ell],[P])\in \Fl(2,3,H),\  P\cap Z=\ell\ \text{(set-theoretically)}\}, 
$$
and from the exact sequence (\ref{ex_seq_taut_bundles_2_3}):
$$
  0\longrightarrow e^*\mathcal O_{G(2,H)}(-1)\otimes t^*\mathcal O_{G(3,H)}(1)\longrightarrow t^*\mathcal E_3\longrightarrow e^*\mathcal E_2\longrightarrow 0,
$$
we see that $e^*\mathcal O_{G(2,H)}(-1)\otimes t^*\mathcal O_{G(3,H)}(1)$ is, for $([\ell],[P])\in \Fl(2,3,H)$, the bundle of equations of $\ell\subset P$. As a result, $F_0(Z)$ is the zero locus on $\Fl(2,3,H)$ of a section of the rank $9$ vector bundle $\mathcal F$ defined by the exact sequence 
\begin{equation}\label{ex_seq_def_F}
    0\longrightarrow e^*\mathcal O_{G(2,H)}(-3)\otimes t^*\mathcal O_{G(3,H)}(3)\longrightarrow t^*\Sym^3\mathcal E_3\longrightarrow \mathcal F\longrightarrow 0.\end{equation}
In particular (since $\mathcal F$ is globally generated by the sections induced by $H^0(t^*\Sym^3\mathcal E_3)$), by Bertini-type theorems, for $Z$ general, $F_0(Z)$ is a smooth surface with $K_{F_0(Z)}\simeq (t^*\mathcal O_{G(3,H)}(3))|_{F_0(Z)}$.
Its link to the surface of planes of a cubic $5$-fold is the following. 

 \begin{proposition}\label{prop_etale_cover_F_2_F_0} Denoting by $X_Z=\{X_6^3-\eq_Z(X_0,\dots,X_5)=0\}$ the cyclic cubic $5$-fold associated to $Z$, the linear projection with center $p_0:=[0:\cdots:0:1]$ induces a degree $3$ \'etale cover $\pi\colon F_2(X_Z)\rightarrow F_0(Z)$ given by the torsion line bundle $(e^*\mathcal O_{G(2,H)}(-1)\otimes t^*\mathcal O_{G(3,H)}(1))|_{F_0(Z)}$.

   In particular, when $F_0(Z)$ is smooth, $F_2(X_Z)$ and $F_0(Z)$ are smooth and irreducible.
\end{proposition}

 \begin{proof} (1) The point $p_0$ does not belong to $X_Z$. In particular, any $[P]\in F_2(X_Z)$ is sent by $\pi_{p_0}\colon \mathbb P(V^*)\dashrightarrow \mathbb P(H^*)$ to a plane in $\mathbb P(H^*)$, where $V=H\oplus \mathbb C\cdot p_0$. The restriction of $\pi_{p_0}$ (also denoted by $\pi_{p_0}$) to $X$ is a degree $3$ cyclic cover of $\mathbb P^5$ ramified over $Z$. Let us denote by $\tau\colon [a_0:\cdots:a_6]\mapsto[a_0:\cdots:a_5:\xi a_6]$, with $\xi$ a primitive third root of $1$, the cover automorphism.

    For any $[P]\in F_2(X_Z)$, $\pi_{p_0}\colon \pi_{p_0}^{-1}(\pi_{p_0}(P))\rightarrow \pi_{p_0}(P)$ is a degree $3$ cyclic cover ramified over the cubic curve $\pi_{p_0}(P)\cap Z$. It contains the three sections $P, \tau(P),\tau^2(P)$, which in turn all contain (set-theoretically) the ramification curve $\pi_{p_0}(P)\cap Z$, so it is a line; \textit{i.e.}, $([\{\pi_{p_0}(P)\cap Z\}_{\red}],[\pi_{p_0}(P)])\in F_0(Z)$.

     Conversely, for any $([\ell],[P])\in F_0(Z)$, $\pi_{p_0}^{-1}|_{X_Z}(P)\rightarrow P$ is a degree $3$ cyclic cover ramified over $\{\ell\}^3$, so it consists of three surfaces isomorphic each to $P$, \textit{i.e.}, three planes. To make it even more explicit, if $P=\{X_0=X_1=X_2=0\}$ and $\ell=\{X_0=X_1=X_2=X_3=0\}$, then $\pi_{p_0}^{-1}|_{X_Z}(P)$ is defined in $\pi_{p_0}^{-1}(P)\simeq\Span(P,p_0)\simeq \mathbb P^3$ by $X_6^3-aX_3^3$ for some $a\neq 0$ (since $Z$ contains no plane), and we have $X_6^3-aX_3^3=(X_6-bX_3)(X_6-b'X_3)(X_6-b''X_3)$, where $b,b',b''$ are the distinct roots of $y^3=a$. So $\pi\colon F_2(X_Z)\rightarrow F_0(Z)$ is \'etale of degree $3$.

     (2) The equation $\eq_Z$ defines a section $\sigma_{\eq_Z}\in H^0(t^*\Sym^3\mathcal E_3)\simeq H^0(\Sym^3\mathcal E_3)$ and by projection in (\ref{ex_seq_def_F}) a section $\overline{\sigma_{\eq_Z}}$ of $\mathcal F$ whose zero locus is $F_0(Z)$. Restricting (\ref{ex_seq_def_F}) to $F_0(Z)$, we see that $\sigma_{\eq_Z}$ induces a section of $(e^*\mathcal O_{G(2,H)}(-3)\otimes t^*\mathcal O_{G(3,H)}(3))|_{F_0(Z)}$ which vanishes nowhere since $Z$ contains no plane. Thus
     $$
     \left(e^*\mathcal O_{G(2,H)}(-3)\otimes t^*\mathcal O_{G(3,H)}(3)\right)|_{F_0(Z)}\simeq \mathcal O_{F_0(Z)}.
     $$

     Now if $(e^*\mathcal O_{G(2,H)}(-1)\otimes t^*\mathcal O_{G(3,H)}(1))|_{F_0(Z)}\simeq \mathcal O_{F_0(Z)}$, since $(e^*\mathcal O_{G(2,H)}(-1)\otimes t^*\mathcal O_{G(3,H)}(1))|_{F_0(Z)}$ is the bundle of equation of $\ell_x\subset P_x$ for any $x=([\ell_x],[P_x])\in F_0(Z)$, for any nowhere-vanishing section $s$ of $(e^*\mathcal O_{G(2,H)}(-1)\otimes t^*\mathcal O_{G(3,H)}(1))|_{F_0(Z)}$, we would be able to define three distinct sections of $\pi\colon F_2(X_Z)\rightarrow F_0(Z)$, namely (symbolically) $[x\mapsto \{X_6-\xi^k s(x)\}_{\Span(P_x,p_0)}]$, $k=0,1,2$. But according to \cite[Proposition 1.8]{Coll_cub}, $F_2(X)$ is connected for any $X$. Hence we have a contradiction. So $(e^*\mathcal O_{G(2,H)}(-1)\otimes t^*\mathcal O_{G(3,H)}(1))|_{F_0(Z)}$ is a non-trivial $3$-torsion line bundle.

      Moreover, we readily see that for any $[P]\in F_2(X_Z)$, $X_{6}|_{P}\neq 0$ is an equation of the line $P\cap \mathbb P(H^*)$; \textit{i.e.}, $\pi^*(e^*\mathcal O_{G(2,H)}(-1)\otimes t^*\mathcal O_{G(3,H)}(1))|_{F_0(Z)}$ has a nowhere-vanishing section, hence is trivial.

 (3) When $F_0(Z)$ is smooth, since $\pi$ is \'etale, $F_2(X_Z)$ is also smooth. As $F_2(X_Z)$ is connected (by \cite[Proposition 1.8]{Coll_cub}), $F_2(X_Z)$ is irreducible, and $\pi(F_2(X_Z))=F_0(Z)$ is also irreducible.
\end{proof}

\begin{remarque}\label{rmk_result_GK} That $F_0(Z)$ is smooth and irreducible, for $Z$ general, is proven in \cite[Lemma 4.3]{GK_geom_lines} without reference to $F_2(X_Z)$.
\end{remarque}

In \cite{GK_geom_lines}, the interest for the image $e(F_0(Z))\subset F_1(Z)$ stems from $e(F_0(Z))$ being the fixed locus of a rational self-map of the hyper-K\"ahler $4$-fold $F_1(Z)$ defined by Voisin (\textit{cf.}~\cite{Voisin_map}).

\begin{proposition}\label{prop_image_F_0} For $Z$ general, the tangent map of $e_{F_0}:=e|_{F_0(Z)}\colon F_0(Z)\rightarrow F_1(Z)$ is injective, and $e_{F_0}$ is the normalisation of $e_{F_0}(F_0(Z))$ and is an isomorphism unto its image outside a finite subset of $F_0(Z)$.
  
 Moreover, $e_{F_0}(F_0(Z))$ is a $($non-normal\,$)$ Lagrangian surface of the hyper-K\"ahler $4$-fold $F_1(Z)$.
\end{proposition}

\begin{proof} (1) That $e_{F_0}$ is injective outside a finite number of points follows from a simple dimension count: let us introduce $I:=\{(([\ell],[P]),[Z])\in \Fl(2,3,H)\times |\mathcal O_{\mathbb P^5}(3)|,\ \ell\subset Z\ \text{and}\ Z\cap P=\ell\ \text{set-theoretically}\}$ and $I_2:=\{(([\ell],[P_1],[P_2]),[Z])\in \mathbb P(\mathcal Q_2)\times_{G(2,H)}\mathbb P(\mathcal Q_2)\backslash \Delta_{\mathbb P(\mathcal Q_2)}\times |\mathcal O_{\mathbb P^5}(3)|,\ \ell\subset Z\ \text{and}\ Z\cap P_i=\ell,\ i=1,2\ \text{set-theoretically}\}$. As $\Fl(2,3,H)$ and $\mathbb P(\mathcal Q_2)\times_{G(2,H)}\mathbb P(\mathcal Q_2)\backslash \Delta_{\mathbb P(\mathcal Q_2)}$ are homogeneous, the fibers of $p\colon I\rightarrow \Fl(2,3,H)$ (resp.\ $p_2\colon I_2\rightarrow \mathbb P(\mathcal Q_2)\times_{G(2,H)}\mathbb P(\mathcal Q_2)\backslash \Delta_{\mathbb P(\mathcal Q_2)}$) are isomorphic to each other and are sub-linear systems of $|\mathcal O_{\mathbb P^5}(3)|$.
  
   Notice that, since $F_0(Z)$ is a surface for $Z$ general, we know that $\dim(I)=\dim(|\mathcal O_{\mathbb P^5}(3)|+2$.

    Let us analyse the fiber of $p_2$. To do so, we can assume $\ell=\{X_2=\cdots=X_5=0\}$, $P_1=\{X_3=X_4=X_5=0\}$ and $P_2=\{X_2=X_4=X_5=0\}$. Then the condition $Z\cap P_1=\ell$ implies that $\eq_Z$ is of the form 
    \begin{equation}\label{normal_form_oscul}\eq_Z=\alpha X_2^3 + X_3Q_3 + X_4Q_4 +X_5Q_5 + \sum_{i=0}^2X_iD_i(X_3,X_4,X_5) + R(X_3,X_4,X_5),
\end{equation} 
where the $Q_i(X_0,X_1,X_2)$ are quadratic forms in $X_0,X_1,X_2$, the $D_i$ are quadratic forms in $X_3,X_4,X_5$ and $R$ is a cubic form in $X_3,X_4,X_5$. Notice that this is the general form of a member of the fiber $p^{-1}([\ell],[P_1])$, in particular, $\dim(p^{-1}([\ell],[P_1]))=\dim(|\mathcal O_{\mathbb P^5}(3)|)+2-\dim(\Fl(2,3,H))=\dim(|\mathcal O_{\mathbb P^5}(3)|)-9$.

 The additional condition $Z\cap P_2=\ell$ implies that $Q_3(X_0,X_1,0)=0$, $D_0(X_3,0,0)=0$, $D_1(X_3,0,0)=0$, which gives $3+1+1=5$ constraints. So  $\dim(p_2^{-1}(([\ell],[P_1],[P_2])))=\dim(p^{-1}([\ell],[P_1]))-5=\dim(|\mathcal O_{\mathbb P^5}(3)|)-14$, hence $\dim(I_2)=\dim(p_2^{-1}(([\ell],[P_1],[P_2])))+2\times 3+\dim(G(2,H))=\dim(|\mathcal O_{\mathbb P^5}(3)|)$. As a result, the general fiber of $I_2\rightarrow |\mathcal O_{\mathbb P^5}(3)|$ is finite. In other words, for $[Z]\in |\mathcal O_{\mathbb P^5}(3)|$ general, there are only finitely many $\ell\subset Z$ such that there are at least two planes $P_1,P_2\subset \mathbb P^5$ such that $Z\cap P_i=\ell$, $i=1,2$, \textit{i.e.}, there is a finite set $\gamma\subset F_0(Z)$ such that $e|_{F_0}\colon F_0(Z)\backslash \gamma\rightarrow F_1(Z)$ is a bijection unto its image.

 (2) Let us give a description of $T_{F_0(Z),([\ell],[P])}$. We recall that the two projective bundle structures on $\Fl(2,3,H)$ given by $e\colon \Fl(2,3,H)\simeq \mathbb P(\mathcal Q_2)\rightarrow G(2,H)$ and $t\colon \Fl(2,3,H)\simeq \mathbb P(\wedge^2\mathcal E_3)\rightarrow G(3,H)$ yield the following descriptions of the tangent bundle:
 $$
 T_{\Fl(2,3,H),([\ell],[P])}\simeq \Hom(\langle\ell\rangle,H/\langle\ell\rangle)\oplus \Hom(\langle P\rangle/\langle\ell\rangle,H/\langle P\rangle)$$ and $$T_{\Fl(2,3,H), ([\ell],[P])}\simeq \Hom(\langle P\rangle,H/\langle P\rangle)\oplus \Hom(\langle \ell\rangle,\langle P\rangle/\langle \ell\rangle).
 $$
The isomorphism between the two takes the  form
\begin{align*}
\Hom\left(\langle\ell\rangle,H/\langle\ell\rangle\right)\oplus \Hom\left(\langle P\rangle/\langle\ell\rangle,H/\langle P\rangle\right) & \longrightarrow \Hom\left(\langle P\rangle,H/\langle P\rangle\right)\oplus \Hom\left(\langle \ell\rangle,\langle P\rangle/\langle \ell\rangle\right),\\
\left(\varphi,\psi\right) & \longmapsto \left(\varphi_\perp+\psi,\varphi_\parallel\right)
\end{align*}
where $\varphi=(\varphi_\parallel,\varphi_\perp)$ is the decomposition corresponding to the choice of a decomposition $H/\langle\ell\rangle\simeq \langle P\rangle/\langle\ell\rangle\oplus H/\langle P\rangle$ coming from a decomposition $\langle P\rangle\simeq \langle\ell\rangle\oplus \langle P\rangle/\langle\ell\rangle$.

Around $([\ell],[P])\in F_0(Z)$, the points of $\Fl(2,3,H)$ are of the form $([(\id_{\langle\ell\rangle}+\varphi)(\langle\ell\rangle)],[(\id_{\langle P\rangle}+\varphi_\perp+\psi)(\langle P\rangle)])$. Let us choose an equation $\lambda\in \langle P\rangle^*$ (a generator of $(\langle P\rangle/\langle\ell\rangle)^*$) of $\ell\subset P$ such that $\eq_Z(x,x,x)=\lambda(x)^3$ for any $x\in \langle P\rangle$.

 The first-order deformation of this equation to an equation of $(\id_{\langle\ell\rangle}+\varphi)(\langle\ell\rangle)\subset (\id_{\langle P\rangle}+\varphi_\perp+\psi)(\langle P\rangle)$ is given by $\lambda-\varphi^*(\lambda)$, so that the point associated to $(\varphi,\psi)$ belongs to $F_0(Z)$ if and only if 
 $$
 \eq_Z(x+\varphi_\perp(x)+\psi(x),x+\varphi_\perp(x)+\psi(x),x+\varphi_\perp(x)+\psi(x))=(1+c(\varphi,\psi))(\lambda(x)-\varphi^*(\lambda)(x))^3\quad \forall x\in \langle P\rangle
 $$
 for some term $c(\varphi,\psi)=O(\varphi,\psi)$ constant on $\langle P\rangle$. So at the first order, we get 
\begin{equation}\label{descript_T_F_0} \eq_Z(x,x,\varphi_\perp(x)+\psi(x))=-\lambda(x)^2\varphi^*(\lambda)(x)+\frac 1 3 c(\varphi,\psi)\lambda(x)^3\quad \forall x\in\langle P\rangle.
\end{equation}
The differential of the projection $e_{F_0(Z)}\colon F_0(Z)\rightarrow F_1(Z)$ is simply given by $(\varphi,\psi)\mapsto \varphi$.

 Let us introduce
$$
J:=\{(([\ell],[P]),[Z])\in \Fl(2,3,H)\times |\mathcal O_{\mathbb P^5}(3)|,\ \ell\subset Z,\ Z\cap P=\ell\ \text{and}\ T_{([\ell],[P])}e|_{F_0} \text{is\ not\ injective}\}
$$
and analyse the fibers of $p_J\colon J\rightarrow \Fl(2,3,H)$, which are isomorphic to each other by the homogeneity of $\Fl(2,3,H)$.

So we can assume $\ell=\{X_2=\cdots=X_5=0\}$ and $P=\{X_3=\cdots=X_5=0\}$, so that $\eq_Z$ is of the form (\ref{normal_form_oscul}) with $Q_i=a_iX_0^2+b_iX_1^2+c_iX_2^2+d_iX_0X_1+e_iX_0X_2+f_iX_1X_2$, $i=3,4,5$, for some $a_i,\dots,f_i$. We recall that for $\varphi=\left(\begin{smallmatrix}u_2 &v_2\\ u_3 &v_3\\ u_4 &v_4\\ u_5 &v_5\end{smallmatrix}\right)\in \Hom(\langle\ell\rangle,H/\langle\ell\rangle)$ and $\psi=\left(\begin{smallmatrix}w_3\\ w_4\\ w_5\end{smallmatrix}\right)\in \Hom(\langle P\rangle/\langle\ell\rangle,H/\langle P\rangle)$, the associated subspaces are
    $$
    \ell_{(\varphi,\psi)}=[\lambda,\mu,\lambda u_2+\mu v_2,\dots,\lambda u_5+\mu v_5],\quad [\lambda,\mu]\in \mathbb P^1, 
    $$
    $$
    P_{(\varphi,\psi)}=[\lambda,\mu,\nu,\lambda u_3+\mu v_3+\nu w_3, \lambda u_4+\mu v_4+\nu w_4,\lambda u_5+\mu v_5+\nu w_5],\quad [\lambda,\mu,\nu]\in \mathbb P^2.
    $$

 Now, if $(0,\psi)\in T_{F_0(Z),([\ell],[P])}$, we have at the first order 
 $$
 \begin{aligned} \eq_{Z}|_{P_{(0,\psi)}}&=\alpha\nu^3 +\sum_{i=3}^5\nu w_i(a_i\lambda^2+b_i\mu^2+c_i\nu^2+d_i\lambda\mu+e_i\lambda\nu+f_i\mu\nu)+O\left(\left(\varphi,\psi\right)^2\right)\\
&=\left(\alpha+c_3w_3+c_4w_4+c_5w_5\right)\nu^3+\left[\left(e_3w_3+e_4w_4+e_5w_5\right)\lambda+\left(f_3w_3+f_4w_4+f_5w_5\right)\mu\right]\nu^2\\
&\ \ +\left(a_3w_3+a_4w_4+a_5w_5\right)\lambda^2\nu +\left(b_3w_3+b_4w_4+b_5w_5\right)\mu^2\\
   &\ \ +\left(d_3w_3+d_4w_4+d_5w_5\right)\lambda\mu\nu+ O\left(\left(\varphi,\psi\right)^2\right)\end{aligned}
 $$
so that looking at (\ref{descript_T_F_0}), we see that $(0,\psi)\in T_{F_0(Z),([\ell],[P])}$ if and only if
$$
\rank\left(\begin{smallmatrix}a_3 &a_4 &a_5\\ b_3 &b_4 &b_5\\ d_3 &d_4 &d_5\\ e_3 &e_4 &e_5\\ f_3 &f_4 &f_5\end{smallmatrix}\right)\leq 2, 
  $$
  which defines a subset of codimension $(3-2)(5-2)=3$.

   So $J\subset I$ has codimension $3$. As $\dim(I)=\dim(|\mathcal O_{\mathbb P^5}(3)|)+2$, $J$ does not dominate $|\mathcal O_{\mathbb P^5}(3)|$; \textit{i.e.}, for the general $Z$, $e_{F_0}$ is an immersion.

(3) Let us prove that $e_{F_0}(F_0(Z))$ is a Lagrangian surface of $F_1(Z)$. In \cite{Iliev-Manivel_cub_hyp_int_syst}, the following explicit description of the symplectic form $\mathbb C\cdot \Omega=H^{2,0}(F_1(Z))$ is given: let us introduce the following quadratic form on $\wedge^2T_{F_1(Z),[\ell]}$ with values in $\Hom((\wedge^2\langle\ell\rangle)^{\otimes 2},\wedge^4(H/\langle\ell\rangle)$: 
   $$
   \begin{aligned}
     K(u\wedge v,u'\wedge v')&= u(x)\wedge u'(y)\wedge v(x)\wedge v'(y) - u(y)\wedge u'(y)\wedge v(x)\wedge v'(x)\\
     &\ \ \ +u(y)\wedge u'(x)\wedge v(y)\wedge v'(x) - u(x)\wedge u'(x)\wedge v(y)\wedge v'(y), 
   \end{aligned}
   $$
   where $(x,y)$ is a basis of $\langle\ell\rangle$. Let us also introduce the following skew-symmetric form:  
\begin{alignat*}{2}
\omega\colon & \wedge^2T_{F_1(Z),[\ell]} &\ \longrightarrow \ & (\wedge^2\langle\ell\rangle)^{\otimes 3}\\
&  u\wedge v &\ \longmapsto \ &\eq_Z(x,x,u(y))\eq_Z(y,y,v(x)) -\eq_Z(x,x,v(y))\eq_Z(y,y,u(x))\\
&&& \  +2\eq_Z(x,y,u(y))\eq_Z(x,x,v(y))-2\eq_Z(x,x,u(y))\eq_Z(x,y,v(y))\\
&&& \  +2\eq_Z(y,y,u(x))\eq_Z(x,y,v(x))-2\eq_Z(x,y,u(x))\eq_Z(y,y,v(x)).
\end{alignat*}
According to \cite[Theorem 1]{Iliev-Manivel_cub_hyp_int_syst}, for $u,v\in T_{F_1(Z),[\ell]}$,
$$
K(u\wedge v,u\wedge v)=w(u\wedge v)\Omega_{[\ell]}(u,v).
$$
As for a general point $([\ell],[P])\in F_0(Z)$, $\ell\subset Z$ is of the first type; \textit{i.e.}, in reference to the above presentation (\ref{normal_form_oscul}) for $\ell=\{X_2=\cdots=X_5=0\}$, $P=\{X_3=X_4=X_5=0\}$, $\left|\begin{smallmatrix}a_3 &b_3 &d_3\\
a_4 &b_4 &d_4\\ a_5 &b_5 &d_5\end{smallmatrix}\right|\neq 0$, it is sufficient to prove the vanishing of $\Omega_{[\ell]}(\Im(T_{([\ell],[P])}e_{F_0}),\Im(T_{([\ell],[P])}e_{F_0}))$ for such a line. So we can assume $\alpha=1$ and
$$
\begin{aligned}Q_3&=X_0^2+e_3X_0X_2+f_3X_1X_2+c_3X_2^2,\\
Q_4&=X_0X_1+e_4X_0X_2+f_4X_1X_2+c_4X_2^2,\\
Q_5&=X_1^2+e_5X_0X_2+f_5X_1X_2+c_5X_2^2.
\end{aligned}
$$ 
Then as above, for $\varphi=\left(\begin{smallmatrix}u_2 &v_2\\ u_3 &v_3\\ u_4 &v_4\\ u_5 &v_5\end{smallmatrix}\right)\in \Hom(\langle\ell\rangle,H/\langle\ell\rangle)$ and $\psi=\left(\begin{smallmatrix}w_3\\ w_4\\ w_5\end{smallmatrix}\right)\in \Hom(\langle P\rangle/\langle\ell\rangle,H/\langle P\rangle)$, we have
    $$
    \begin{aligned}
      \eq_{Z}|_{P_{(\varphi,\psi)}} &= \nu^3+\sum_{i=3}^5(\lambda u_i+\mu v_i+\nu w_3)Q_i+O((\varphi,\psi)^2)\\
&=(1+c_3w_3+c_4w_4+c_5w_5)\nu^3+(c_3u_3+e_3w_3+c_4u_4+e_4w_4+c_5u_5+e_5w_5)\lambda\nu^2\\
&\ \ + (c_3v_3+f_3w_3+c_4v_4+f_4w_4+c_5v_5+b_5w_5)\mu\nu^2\\
&\ \ + (w_3+e_3u_3+e_4u_4+e_5u_5)\lambda^2\nu +(w_5+f_3v_3+f_4v_4+f_5v_5)\mu\nu^2\\
&\ \ +(w_4+f_3u_3+e_3v_3+f_4u_4+e_4v_4+f_5u_5+e_5v_5)\lambda\mu\nu\\
      &\ \ +u_3\lambda^3+v_5\mu^2+(v_4+u_5)\lambda\mu^2 +(v_3+u_4)\lambda^2\mu + O\left((\varphi,\psi)^2\right),
    \end{aligned}$$
so that the description (\ref{descript_T_F_0}) of $T_{F_0(Z),([\ell],[P])}$ yields
$$\left\{\begin{aligned} c_3u_3+e_3w_3+c_4u_4+e_4w_4+c_5u_5+e_5w_5 &=-u_2\\
c_3v_3+f_3w_3+c_4v_4+f_4w_4+c_5v_5+b_5w_5 & = - v_2\\
w_3+e_3u_3+e_4u_4+e_5u_5 &= 0\\
w_5+f_3v_3+f_4v_4+f_5v_5 &=0\\
w_4+f_3u_3+e_3v_3+f_4u_4+e_4v_4+f_5u_5+e_5v_5 &=0\\
v_4 =-u_5;\ v_3 =-u_4\ u_3=0\ v_5 &=0.
\end{aligned} \right.
$$
The seven last equations yield $w_3=-(e_4u_4+e_5u_5)$, $w_4=(e_3-f_4)u_4+(e_4-f_5)u_5$, $w_5=f_3u_4+ f_4u_5$. Thus the first two give a system
$$\left\{\begin{aligned}\alpha u_4 + \beta u_5=-u_2,\\ -\delta u_4- \alpha u_5=-v_2,\end{aligned} \right.$$ 
 where $\alpha=c_4-e_4f_4+e_5f_3$, $\beta=c_5-e_3e_5+e_4^2-e_4f_5+e_5f_4$ and $\delta=e_3f_4-f_4^2-e_4f_3+f_3f_5-c_3$. In particular, the determinant $\Delta=-\alpha^2-\beta\delta$ of the $2\times 2$ system is non-zero for a general choice of the $(e_i,f_i,c_i)$ and
$u_4=\frac{1}{\Delta}(\alpha u_2+\beta v_2)$, 
$u_5=\frac{1}{\Delta}(\delta u_2-\alpha v_2)$.
So a basis of $T_{F_1(Z),([\ell],[P])}$ is given by ($(u_2=1,v_2=0)$ and $(u_2=0,v_2=1)$): 
\begin{alignat*}{2}
\varphi_{u_2}\colon  &\epsilon_0 &\ \longmapsto \ &\epsilon_2 +\frac{\alpha}{\Delta}\epsilon_4 +\frac{\delta}{\Delta}\epsilon_5,\\
  &\epsilon_1 &\ \longmapsto \ &-\frac{\alpha}{\Delta}\epsilon_3-\frac{\delta}{\Delta}\epsilon_4
\end{alignat*}
and 
\begin{alignat*}{2}
\varphi_{v_2}\colon  &\epsilon_0 &\ \longmapsto \ &\frac{\beta}{\Delta}\epsilon_4 -\frac{\alpha}{\Delta}\epsilon_5,\\
  &\epsilon_1 &\ \longmapsto \ &\epsilon_2-\frac{\beta}{\Delta}\epsilon_3+\frac{\alpha}{\Delta}\epsilon_4, 
\end{alignat*}
where $(\epsilon_0,\dots,\epsilon_5)$ is the (dual) basis associated to the choice of the coordinates $X_i$. Then we readily compute $$K(\varphi_{u_2}\wedge\varphi_{v_2})=\left|\begin{matrix}1 &0 &0 &1\\ 0 &-\frac{\alpha}{\Delta} &0 &-\frac{\beta}{\Delta}\\ \frac{\alpha}{\Delta} &-\frac{\delta}{\Delta} &\frac{\beta}{\Delta} &\frac{\alpha}{\Delta}\\ \frac{\delta}{\Delta} &0 &-\frac{\alpha}{\Delta} &0\end{matrix}\right|=0$$
and $\omega(\varphi_{u_2}\wedge\varphi_{v_2})=\frac{5}{\Delta}\neq 0$, hence $\Omega_{[\ell]}(\varphi_{u_2},\varphi_{v_2})=0$.
\end{proof}

\begin{remarque}\label{rmk_result_GK2} In \cite{GK_geom_lines}, it is also proven that $F_0(Z)\rightarrow e(F_0(Z))$ is the normalisation and that $e(F_0(Z))$ has $3780$ non-normal isolated singularities.
\end{remarque}

As for $Z$ general, $e_{F_0}$ is an immersion, $N_{F_0(Z)/F_1(Z)}:=e_{F_0}^*T_{F_1(Z)}/T_{F_0(Z)}$ is locally free. Moreover, since $e_{F_0}$ is, outside a codimension $2$ subset of $F_0(Z)$, an isomorphism unto its image and that image is a Lagrangian subvariety of $F_1(Z)$, we get (outside a codimension $2$ subset, thus globally) an isomorphism $$\Omega_{F_0(Z)}\simeq N_{F_0(Z)/F_1(Z)}.$$
\indent Notice that $F_0(Z)$ naturally lives in $\mathbb P(\mathcal Q_{2}|_{F_1(Z)})\subset \Fl(2,3,H)$. We have the following. 

\begin{lemme}\label{lem_normal_bdl_intermediate} The following sequence is exact:
  $$
  \begin{aligned}0\longrightarrow e_{F_1}^*\mathcal O_{F_1(Z)}(-3)\otimes t_{F_1}^*\mathcal O_{G(3,H)}(3)\longrightarrow e_{F_1}^*\mathcal O_{F_1}(-1)\otimes t_{F_1}^*(\Sym^2\mathcal E_3\otimes\mathcal O_{G(3,H)}(1))|_{F_1(Z)}\\\longrightarrow N_{F_0(Z)/\mathbb P(\mathcal Q_{2}|_{F_1(Z)})}\longrightarrow 0,\end{aligned}
    $$
    where $e_{F_1}\colon \mathbb P(\mathcal Q_{2}|_{F_1(Z)})\rightarrow F_1(Z)$ and $t_{F_1}\colon \mathbb P(\mathcal Q_{2}|_{F_1(Z)})\rightarrow G(3,H)$. 
\end{lemme}

\begin{proof} We have seen that $F_0(Z)\subset \Fl(2,3,H)$ is the zero locus of a section of $\mathcal F$ appearing in the sequence (\ref{ex_seq_def_F}). Taking the symmetric power of (\ref{ex_seq_taut_bundles_2_3}), we have the following commutative diagram with exact rows:  
  $$
  \xymatrix@-1pc{%@M=-0.1em{
    0\ar[r] &e^*\mathcal O_{G(2,H)}(-3)\otimes t^*\mathcal O_{G(3,H)}(3)\ar[r]\ar[d] &e^*\mathcal O_{G(2,H)}(-3)\otimes t^*\mathcal O_{G(3,H)}(3)\ar[r]\ar[d] &0\ar[r]\ar[d] &0\\
    0\ar[r] &e^*\mathcal O_{G(2,H)}(-1)\otimes t^*(\Sym^2\mathcal E_3\otimes \mathcal O_{G(3,H)}(1))\ar[r] &t^*\Sym^3\mathcal E_3\ar[r] &e^*\Sym^2\mathcal E_2\ar[r] &0.}
  $$
  The projection to $e^*\Sym^2\mathcal E_2$ of the section $\sigma_{\eq_Z}\in H^0(t^*\Sym^3E_3)$ induced by $\eq_Z$ vanishes on $F_1(Z)$ by the definition of $F_1(Z)$. So it induces a section of
  $$e_{F_1}^*\mathcal O_{F_1(Z)}(-1)\otimes t_{F_1}^*(\Sym^2\mathcal E_3\otimes \mathcal O_{G(3,H)}(1))\simeq (e^*\mathcal O_{G(2,H)}(-1)\otimes t^*(\Sym^2\mathcal E_3\otimes \mathcal O_{G(3,H)}(1)))|_{\mathbb P(\mathcal Q_{2}|_{F_1(Z)})}.
  $$
  Now the snake lemma in the above diagram gives the result.
\end{proof}

The snake lemma in the diagram with exact rows
$$\xymatrix@-1pc{%@M=-0.1em{
  0\ar[r] &T_{F_0(Z)}\ar[r]\ar[d]^{\cong} &T_{\mathbb P(\mathcal Q_{2}|_{F_1(Z)})|_{F_0(Z)}}\ar[r]\ar[d] &N_{F_0(Z)/\mathbb P(\mathcal Q_{2}|_{F_1(Z)})}\ar[r]\ar[d] &0\\
0\ar[r] &T_{F_0(Z)}\ar[r] &e_{F_0}^*T_{F_1(Z)}\ar[r] &N_{F_0(Z)/F_1(Z)}\ar[r] &0}$$
and the description of the relative tangent bundle of $e_{F_1}$ give the following. 

\begin{proposition}\label{prop_cotgt_bundle_ex_seq_F_0} The following sequence is exact: 
  $$
  \begin{aligned}0\longrightarrow \mathcal O_{F_0}\longrightarrow e_{F_0}^*(\mathcal Q_{2}|_{F_1(Z)}\otimes \mathcal O_{F_1(Z)}(-1))\otimes t_{F_0}^*(\mathcal O_{G(3,H)}(1))|_{F_0}\longrightarrow N_{F_0(Z)/\mathbb P(\mathcal Q_{2}|_{F_1(Z)})}\longrightarrow \Omega_{F_0(Z)}\longrightarrow 0.\end{aligned}
  $$
\end{proposition}

We finish this section by computing the Hodge numbers of $F_0(Z)$.

\begin{proposition}\label{prop_h_1_F_0} We have $H^1(F_0(Z),\mathbb Z)=0$ for any $Z$ for which $F_0(Z)$ is smooth.
\end{proposition}

\begin{proof} For the universal variety of planes $r_{\univ}\colon \mathcal F_2(\mathcal X)\rightarrow |\mathcal O_{\mathbb P^6}(3)|$, $R^3r_{\univ,*}\mathbb Q$ is a local system over the open subset $\{[X]\in |\mathcal O_{\mathbb P^6}(3)|,\ F_2(X)\ \text{is smooth}\}$ which, by Proposition \ref{prop_etale_cover_F_2_F_0}, contains an open subset of the locus of cyclic cubic $5$-folds.
  
 As a consequence, the Abel--Jacobi isomorphism $q_*p^*\colon H^3(F_2(X),\mathbb Q)\shortisor H^5(X,\mathbb Q)$ given by the result of Collino (Theorem \ref{thm_Collino_intro}) for general $X$ extends to the case of the general cyclic cubic $5$-fold.

 But, as noticed in the proof of Proposition \ref{prop_etale_cover_F_2_F_0}, for any $[P]\in F_0(Z)$, the associated cycle $q(p^{-1}(\pi^{-1}([P])))$ on $X_Z$ is the complete intersection cycle $\Span(P,p_0)\cap X_Z$, which belongs to a family of cycles parametrised by a rational variety, namely $\{[\Pi]\in G(4,V),\ p_0\in \Pi\}\simeq G(3,H)$. Now, as an abelian variety contains no rational curve, the Abel--Jacobi map $\Phi\colon G(3,H)\rightarrow J^5(X_Z)$, $[P]\mapsto [\Span(P,p_0)\cap X_Z] - [\Span(P_0,p_0)\cap X_Z]$ ($[P_0]$ being a reference point) is constant. Hence the restriction $\Phi_{(\pi_*,\id_{X_Z})\mathbb P(\mathcal E_3)}\colon F_0(Z)\rightarrow J^5(X_Z)$ of $\Phi$ to the sub-family $(\pi_*,\id_{X_Z})\mathbb P(\mathcal E_3)\subset F_0(Z)\times X_Z$ (of planes $P$ such that $\Span(P,p_0)\cap X_Z$ consists of three planes) is constant; \textit{i.e.}, $q_*p^*\pi^*\colon H^3(F_0(Z),\mathbb Z)\rightarrow H^5(X_Z,\mathbb Z)$ is trivial.

 As $\pi$ is \'etale, $\pi^*\colon H^3(F_0(Z),\mathbb Q)\rightarrow H^3(F_2(X_Z),\mathbb Q)$ is injective, so that the trivial map $q_*p^*\pi^*$ is the composition of a injective map followed by an isomorphism. 
\end{proof}

We can then compute the rest of the Hodge numbers:
\begin{enumerate}
\item Again using the package Schubert2 of Macaulay2, we can use the Koszul resolution of $\mathcal O_{F_0(Z)}$ by $\wedge^i\mathcal F^*$ (where $\mathcal F$ is defined by (\ref{ex_seq_def_F})) to compute $\chi(\mathcal O_{F_0(Z)})=1071$ with the following code:
\begin{verbatim}
loadPackage "Schubert2"
G=flagBundle{3,3}
(Q,E)=bundles G
wE=exteriorPower(2,E)
P=projectiveBundle' wE
p=P.StructureMap
pl=exteriorPower(3,E)
pol=p^*pl**dual(OO_P(1))
F=p^*symmetricPower(3,E)-symmetricPower(3,pol)
chi(exteriorPower(0,dual(F)))-chi(exteriorPower(1,dual(F)))
+chi(exteriorPower(2,dual(F)))-chi(exteriorPower(3,dual(F)))
+chi(exteriorPower(4,dual(F)))-chi(exteriorPower(5,dual(F)))
+chi(exteriorPower(6,dual(F)))-chi(exteriorPower(7,dual(F)))
+chi(exteriorPower(8,dual(F)))-chi(exteriorPower(9,dual(F)))
\end{verbatim}
so we get $h^2(\mathcal O_{F_0(Z)})=1070$.\\

\item Then as $\pi$ is \'etale of degree $3$, we get $\chi_{\topp}(F_0(Z))=\frac 1 3 \chi_{\topp}(F_2(X_Z))=4347$. So $h^{1,1}(F_0(Z))=2207$.
\end{enumerate}

%%%%%%%%%%%%%%%%%%%%%
% References
%%%%%%%%%%%%%%%%%%%%%


\begin{thebibliography}{Huy23+++}

\bibitem[CG72]{Cl-Gr}
C.\,H.~Clemens and P.\,A.~Griffiths,
\emph{The intermediate Jacobian of the cubic threefold},
Ann.\ of Math.~(2) \textbf{95} (1972), 281--356.

\bibitem[Col86]{Coll_cub}
A.~Collino,
\emph{The Abel-Jacobi isomorphism for the cubic fivefold},
Pacific J.\ Math.\ \textbf{122} (1986), 43--55.

\bibitem[Gam]{gammel}
  S.~Gammelgaard,
  \emph{A small note about Hilbert schemes and Grothendieck rings}.
Available from \url{https://sorengam.github.io/researchGrotRing}.

\bibitem[GK21]{GK_geom_lines}
  F.~Gounelas and A.~Kouvidakis,
  \emph{Geometry of lines on a cubic fourfold}, Int. Math. Res. Not.  (2023), published online, article ID~\href{https://doi.org/10.1093/imrn/rnac160}{rnac160}.
 

\bibitem[Gri69]{griffiths_periods}
  P.~Griffiths,
  \emph{On the periods of certain rational integrals I, II},
  Ann.\ of Math.~(2) \textbf{90} (1969), 460--541.

\bibitem[HT84]{HT_degener}
  J.~Harris and L.\,W.~Tu,
  \emph{On symmetric and skew-symmetric determinantal varieties},
  Topology \textbf{23} (1984), no.~1, 71--84.

\bibitem[Huy23]{Huy_cub}
  D.~Huybrechts,
  \emph{The geometry of cubic hypersurfaces},
Cambridge Stud.\ Adv.\ Math., vol.~206, Cambridge Univ.\ Press, Cambridge, 2023.
 Draft available from \url{https://www.math.uni-bonn.de/people/huybrech/Notes.pdf}.
  
\bibitem[IM08]{Iliev-Manivel_cub_hyp_int_syst}
  A.~Iliev and L.~Manivel,
  \emph{Cubic hypersurfaces and integrable systems},
  Amer.~J.\ Math.\ \textbf{130} (2008), no.~6, 1445--1475.

\bibitem[Jia12]{jiang_noether_lefschetz}
  Z.~Jiang,
  \emph{A Noether-Lefschetz theorem for varieties of $r$-planes in complete intersections},
  Nagoya Math.~J.\ \textbf{206} (2012), 39--66.

\bibitem[SV16]{shen-vial}
  M.~Shen and C.~Vial,
  \emph{The Fourier transform for certain hyper-K\"ahler fourfolds},
  Mem.\ Amer.\ Math.\ Soc.\ \textbf{240} (2016), no.~1139.
 
\bibitem[Spa03]{spandaw}
  J.~Spandaw,
  \emph{Noether-Lefschetz Problems for Degeneracy Loci},
  Mem.\ Amer.\ Math.\ Soc.\ \textbf{161} (2003), no.~764.
  
\bibitem[Voi92]{voisin_lagrang_cubic_3_fold}
 C.~Voisin,
  \emph{Sur la stabilit\'e des sous-vari\'et\'es lagrangiennes des vari\'et\'es symplectiques holomorphes},
  in: \emph{Complex projective geometry} (Trieste 1989/Bergen 1989), pp.~294--303,   London Math.\ Soc.\ Lecture Note Ser., vol.~179, Cambridge Univ.\ Press, Cambridge, 1992.

\bibitem[Voi04]{Voisin_map}
  \bysame,
  \emph{Intrinsic pseudovolume forms and $K$-correspondences},
  in: \emph{The Fano Conference}, pp. 761--792, Univ.\ Torino, Turin, 2004.
  
\end{thebibliography}
\end{document}